%% file: SFPAinGPA.tex
\documentclass[12pt]{article}
\pdfoutput=1

\usepackage{amsmath, amsthm, amssymb}
 \usepackage{ifpdf}
\ifpdf
\usepackage[pdftex]{graphicx}
\else
\usepackage[dvips]{graphicx}
\fi
\usepackage{tikz}
 	 \usetikzlibrary{arrows,backgrounds} 
\usepackage[all]{xy}

\usepackage[pdftex,plainpages=false,hypertexnames=false,pdfpagelabels]{hyperref}
\newcommand{\arXiv}[1]{\href{http://arxiv.org/abs/#1}{\tt arXiv:\nolinkurl{#1}}}
\newcommand{\doi}[1]{\href{http://dx.doi.org/#1}{{\tt DOI:#1}}}

\newcommand{\googlebooks}[1]{(preview at \href{http://books.google.com/books?id=#1}{google books})}

\theoremstyle{plain}
\newtheorem{thm}{Theorem}[section]
\newtheorem*{thm*}{Theorem}
\newtheorem{cor}[thm]{Corollary}
\newtheorem{lem}[thm]{Lemma}
\newtheorem{fact}[thm]{Fact}
\newtheorem{prop}[thm]{Proposition}
\newtheorem{quest}[thm]{Question}
\theoremstyle{definition}
\newtheorem{defn}[thm]{Definition}
\newtheorem{nota}[thm]{Notation}
\newtheorem{exs}[thm]{Examples}

\newtheorem{rem}[thm]{Remark}

\DeclareMathOperator{\conc}{conc}
\DeclareMathOperator{\conv}{conv}

\DeclareMathOperator{\id}{id}
\DeclareMathOperator{\spann}{span}
\DeclareMathOperator{\Tr}{Tr}
\DeclareMathOperator{\tr}{tr}

\input{pictures/operators}

\newcommand{\D}{\displaystyle}
\newcommand{\comment}[1]{}
\newcommand{\hs}{\hspace{.07in}}
\newcommand{\hsp}[1]{\hs\text{#1}\hs}
\newcommand{\be}{\begin{enumerate}}
\newcommand{\ee}{\end{enumerate}}
 
\newcommand{\itt}[1]{\item[\underline{\text{#1}:}]} 
\newcommand{\N}{\mathbb{N}}
\newcommand{\Z}{\mathbb{Z}} 

\newcommand{\C}{\mathbb{C}}
\newcommand{\I}{\infty} 
\newcommand{\set}[2]{\left\{#1 \middle| #2\right\}}
\newcommand{\thh}{^{\text{th}}}
\newcommand{\E}{\mathcal{E}}
\newcommand{\V}{\mathcal{V}}

\input xy
\xyoption{all}

\begin{document}

\title{The embedding theorem for finite depth subfactor planar algebras}

\author{Vaughan F.R. Jones and David Penneys}

\date{\today}

\maketitle
\begin{abstract}
We define a canonical relative commutant planar algebra from a strongly Markov inclusion of finite von Neumann algebras. In the case of a connected unital inclusion of finite dimensional C$^*$-algebras with the Markov trace, we show this planar algebra is isomorphic to the bipartite graph planar algebra of the Bratteli diagram of the inclusion. Finally, we show that a finite depth subfactor planar algebra is a planar subalgebra of the bipartite graph planar algebra of its principal graph.\end{abstract}
\tableofcontents

\section{Introduction}
A powerful method of construction of subfactors is the use of  \underline{commuting squares}, which are systems of four finite dimensional von Neumann algebras 
\begin{align*}
A_{1,0}  &\subset A_{1,1}\\
\cup\hspace{.07in} &\hspace{.29in}\cup\\
A_{0,0} &\subset A_{0,1}
\end{align*}
included as above, with a faithful trace on $A_{1,1}$ so that $A_{1,0}$ and
$A_{0,1}$ are orthogonal modulo their intersection $A_{0,0}$.

One iterates the basic construction of \cite{MR696688} for the inclusions $A_{i,j}\subset A_{i,j+1}$ and $A_{i,j}\subset A_{i+1,j}$ to obtain a tower of inclusions  $A_{0,n} \subset A_{1,n}$. By a lovely compactness argument of Ocneanu \cite{MR1473221},\cite{MR1642584}, the standard invariant, or higher relative commutants, of the inductive limit inclusion $A_{0,\infty}\subset A_{1,\infty}$ are the algebras $A_{0,1}'\cap A_{n,0}$. Thus once bases have been chosen, the calculation of the relative commutants is a matter of elementary linear algebra. 

It was to formalise this calculation that planar algebras were first introduced \cite{math/9909027}. Finite dimensional inclusions are given by certain graphs (Bratteli diagrams), and in \cite{MR1865703}, a planar algebra associated purely combinatorially to a bipartite graph was introduced so that it is rather obviously the tower of relative commutants for an inclusion $B_0\subset B_1$ having the graph as its Bratteli diagram. But because Ocneanu's notion of \underline{connection} was never completely formalised in \cite{math/9909027}, it was NOT proved that the planar algebra coming from a commuting square via Ocneanu compactness is a planar subalgebra of the one defined in \cite{MR1865703} for the graph of the inclusion $A_{0,0}\subset A_{1,0}$.

Meanwhile the theory of planar algebras grew in its own right and a new method of constructing
subfactors evolved by looking at planar subalgebras of a given planar algebra \cite{0902.1294},\cite{0909.4099}. Now if a subfactor is of \underline{finite depth}, then by \cite{MR1055708}, there is a commuting square that constructs a hyperfinite model of it. Moreover the inclusion $A_{0,0}\subset A_{1,0} $ for this canonical commuting square has Bratteli diagram given by the so-called \underline{principal graph}, which is a powerful subfactor invariant. Thus if the the result of the previous paragraph had been proved, it would have implied the following theorem, which is the main result of this paper:

\begin{thm*}\label{maintheorem} A finite depth subfactor planar algebra is a planar subalgebra of the bipartite graph planar algebra of its principal graph.
\end{thm*}
(See \cite{0808.0764} for the definition of the principal graph of a planar algebra.)  We prove this result with the interesting twist of not using connections. In particular, our proof does not invoke the \underline{dual principal graph}, which is perhaps rather surprising.

There are three steps to our proof. The first step, Section \ref{sec:pa}, is to define a canonical planar algebra structure on the tower of relative commutants from a connected unital inclusion of finite dimensional C$^*$-algebras whose Bratteli diagram is a given graph. We do this in more generality, replacing finite dimensionality by a \underline{strong Markov} property (see Definition \ref{stronglymarkov}), because it is no harder and should have applications. The second step, Section \ref{sec:iso}, is to identify this planar algebra structure with that of \cite{MR1865703} in the finite dimensional case. Finally, in Section \ref{sec:embed}, we construct the embedding map as follows: given a finite depth subfactor planar algebra $Q_\bullet$, pick $s=2r$ suitably large so that the inclusion $Q_{s,+}\subset Q_{s+1,+}\subset (Q_{s+2,+},e_{s+1})$ is standard, i.e., isomorphic to the basic construction. Set $M_0=Q_{s,+}$ and $M_1=Q_{s+1,+}$, and let $P_\bullet$ be the canonical planar algebra $P_\bullet$ associated to the inclusion $M_0\subset M_1$. We prove in Theorem \ref{embed} that the map $Q_\bullet\to P_\bullet$ given by adding $2s$ or $2s+1$ strings on the left, depending whether we are in $Q_{n,+}$ or $Q_{n,-}$ respectively, is an inclusion of planar algebras.
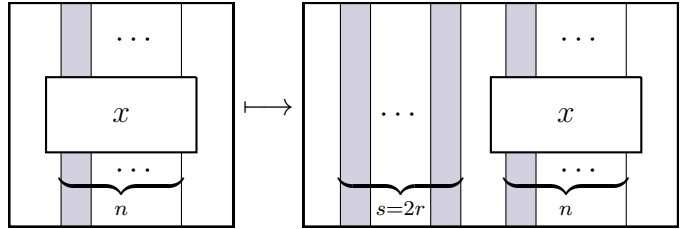
\input{pictures/embed}

While this paper was being written, Morrison and Walker in \cite{gpa} produced a totally different proof which constructs an embedding directly from the planar algebra $Q_\bullet$ without the use of algebra towers and centralisers. Their method also has the advantage that it applies to \underline{infinite depth} subfactor planar algebras without alteration!

Both authors would like to acknowledge support from NSF grants DMS 0401734, DMS 0856316, and UC Berkeley's Geometry, Topology, and Operator Algebras NSF grant EMSW21-RTG.

\section{The canonical planar algebra of a strongly Markov inclusion of finite von Neumann algebras}\label{sec:pa}
After defining the notion of a strongly Markov inclusion of finite von Neumann algebras, we show the basic construction is also strongly Markov with the same (Watatani) index. We then define a canonical planar algebra associated to a strongly Markov inclusion.

Many results of this section can be found in \cite{MR696688}, \cite{MR860811}, \cite{MR996807}, \cite{MR1073519}, \cite{MR1278111}, and \cite{MR1424954}, but we provide some proofs for the reader's convenience.

\subsection{Bases, traces, and strongly Markov inclusions}\label{prelim}

\begin{nota}
Throughout this paper, a trace on a finite von Neumann algebra will mean a faithful, normal, tracial state unless otherwise specified. We will write $M_0\subset (M_1,\tr_1)$ to mean $M_0\subset M_1$ is an inclusion of finite von Neumann algebras where $\tr_1$ is a trace on $M_1$. We set $\tr_0=\tr_1|_{M_0}$.
\end{nota}

Let $M_0\subset (M_1,\tr_1)$. Let $M_2=\langle M_1,e_1\rangle=JM_0'J\subset B(L^2(M_1,\tr_1))$ be the basic construction, where $e_1$ is the Jones projection with range $L^2(M_0,\tr_0)$, and $J\colon L^2(M_1,\tr_1)\to L^2(M_1,\tr_1)$ is the antilinear unitary given by the antilinear extension of $x\Omega\mapsto x^*\Omega$, where $\Omega\in L^2(M_1,\tr_1)$ is the image of $1\in M_1$. The following proposition is straightforward:

\begin{prop}\label{ppbasis}
The following are equivalent for a finite subset $B=\{b\}\subset M_1$:
\item[(i)] $1=\sum\limits_{b\in B} b e_1 b^*$,
\item[(ii)]  $\D x= \sum\limits_{b\in B} b E_{M_0}(b^*x)$ for all $x\in M_1$, and 
\item[(iii)] $\D x= \sum\limits_{b\in B} E_{M_0}(xb) b^*$ for all $x\in M_1$.
\end{prop}

\begin{defn}\label{ppbasisdef} A \underline{Pimsner-Popa basis} for $M_1$ over $M_0$ is a finite subset $B=\{b\}\subset M_1$ for which the conditions in Proposition \ref{ppbasis} hold.  
\end{defn}

We refer the reader to \cite{MR996807} for the proof of the following:
\begin{prop}\label{existence} The following are equivalent:
\item[(i)] There is a Pimsner-Popa basis for $M_1$ over $M_0$,
\item[(ii)] $M_1\otimes_{M_0} M_1\to M_2$ by  $x \otimes y\mapsto x e_1 y$ is an $M_1-M_1$ bimodule isomorphism, and
\item[(iii)] $M_2=M_1 e_1 M_1$.
\end{prop}

\begin{rem}\label{basisproperties}
$M_1\otimes_{M_0} M_1$ is a $*$-algebra with multiplication $(x_1\otimes y_1)(x_2\otimes y_2)=x_1\otimes E_{M_0}(y_1x_2)y_2$ and adjoint $(x\otimes y)^*=y^*\otimes x^*$. If there is a Pimsner-Popa basis for $M_1$ over $M_0$, the sum $\sum_{b\in B} b\otimes b^*$ is independent of the choice of Pimsner-Popa basis $B$, as it is the identity. (We will renormalize in Proposition \ref{renormalize}.) 
\end{rem}

\begin{defn}[\cite{MR996807}]
If there is a Pimsner-Popa basis $B=\{b\}$ for $M_1$ over $M_0$, then we define the (Watatani) index $$[M_1\colon M_0]=\sum_{b\in B}bb^*,$$ which is independent of the choice of basis.
\end{defn}

\begin{defn}\label{canonicaltrace}
Recall from \cite{MR1278111} that $M_2$ has a canonical faithful, normal, semifinite trace $\Tr_2$ which is the extension of the map $x e_1 y\mapsto \tr_1(xy)$ for $x,y\in M_1$.
\end{defn}

\begin{defn}\label{stronglymarkov}
An inclusion $M_0\subset (M_1,\tr_1)$ of finite von Neumann algebras is called \underline{strongly Markov} if 
\item[(1)] $\Tr_2$ is finite with $\Tr_2(1)^{-1}\Tr_2|_{M_1}=\tr_1$ and
\item[(2)] there is a Pimsner-Popa basis for $M_1$ over $M_0$.
\end{defn}

\begin{rem}
Recall from \cite{MR1278111} that $\Tr_2(1)^{-1}\Tr_2$ extends $\tr_1$ if and only if $\Tr_2(1)=[M_1\colon M_0]\in [1,\I)$.
\end{rem}

\begin{exs} 
\item[(1)] A finite Jones index inclusion of $II_1$-factors with the unique trace is strongly Markov, and the Watatani index is equal to the Jones index.
\item[(2)] A connected, unital inclusion of finite dimensional C$^*$-algebras with the Markov trace is strongly Markov, and the index is equal to $\|\Lambda^T\Lambda\|$ where $\Lambda$ is the bipartite adjacency matrix for the Bratteli diagram of the inclusion.
\end{exs}

Suppose $M_0\subset (M_1,\tr_1)$ is strongly Markov. Then $M_2$ is finite and $\tr_2=[M_1\colon M_0]^{-1} \Tr_2$ extends $\tr_1$, so we may iterate the basic construction for $M_1\subset (M_2,\tr_2)$. Let $M_3=\langle M_2,e_2\rangle\subset B(L^2(M_2,\tr_2))$, where $e_2$ is the Jones projection with range $L^2(M_1,\tr_1)$. Let $\Tr_3$ be the canonical faithful, normal, semifinite trace on $M_3$ (see Definition \ref{canonicaltrace}). The following lemma is straightforward:

\begin{lem}\label{TL}
\item[(1)] The conditional expectation $E_{M_1}\colon M_2\to M_1$ is given by $E_{M_1}(x e_1 y)=xy$,
\item[(2)] $e_1 e_2 e_1= [M_1\colon M_0]^{-1}e_1$ and $e_2 e_1 e_2=[M_1\colon M_0]^{-1} e_2$, and
\item[(3)]
if $B$ is a Pimsner-Popa basis for $M_1$ over $M_0$, then $\set{[M_1\colon M_0]^{1/2} b e_1}{b\in B}$ is a Pimsner-Popa basis for $M_2$ over $M_1$.
\end{lem}

\begin{thm} $M_1\subset (M_2,\tr_2)$ is strongly Markov and $[M_2\colon M_1]=[M_1\colon M_0]$.
\end{thm}
\begin{proof}
Note $M_3=M_2 e_2 M_2$ by Proposition \ref{existence} and Lemma \ref{TL}, so the canonical trace $\Tr_3$ on $M_3$ is finite. By Definition \ref{canonicaltrace} and Lemma \ref{TL}, if $x\in M_2$,
{\fontsize{11}{11}{
$$
\Tr_3(x)=[M_1\colon M_0]\sum\limits_{b\in B} \Tr_3(xbe_1e_2e_1b^*)
=[M_1\colon M_0]\sum\limits_{b\in B} \tr_2(xbe_1b^*)=[M_1\colon M_0]\tr_2(x).
$$}}
Hence $[M_2\colon M_1]=\Tr_3(1)=[M_1\colon M_0]$, and $\tr_3=[M_1\colon M_0]^{-1}\Tr_3$ extends $\tr_2$.
\end{proof}

\begin{rem}
Markov inclusions (possibly without Pimsner-Popa bases) have been studied by Jolissaint \cite{MR1073519}, Popa \cite{MR860811}, \cite{MR1278111}, and more. The adjective ``strongly" in the term ``strongly Markov" comes from Definition 3.6 in \cite{MR945550}, where they define the notion of ``fortement d'indice fini" for a conditional expectation. This notion translates as the existence of a finite Pimsner-Popa basis.
\end{rem}

\begin{quest}
It is unknown to the authors at this point whether condition (1) implies condition (2) in Definition \ref{stronglymarkov}. This is the case for connected inclusions with atomic centers \cite{MR1073519}. It is unknown to the authors for connected inclusions with diffuse centers.
\end{quest}

We provide a useful lemma for working with strongly Markov inclusions, which is similar to Lemma 5.8 of \cite{MR1073519} and Lemma 5.3.1 in \cite{MR1473221}.
\begin{lem}\label{basicconstruction}
Suppose $M_0\subset (M_1,\tr_1)$ is strongly Markov. Suppose there is a von Neumann algebra $P\subset B(L^2(M_1,\tr_1))$ containing $M_1$, $\tr_P$ is a trace on $P$ extending $M_1$, and $p\in P$ is a projection such that $E_{M_1}(p)=d^{-2}$. If
\item[(1)] $pwp=E_{M_0}(w)p$ for all $w\in M_1$ and
\item[(2)] the map $M_0\to M_0p$ by $z\mapsto zp$ is injective, 
\item
then $\psi\colon M_1\otimes_{M_0} M_1\to M_1 pM_1$ by $x\otimes y\mapsto x p y$ is an $M_1$-bilinear isomorphism of $*$-algebras. In particular, the map $\widetilde{\psi}\colon M_1pM_1\to M_1e_1M_1=M_2$ by $xp y\mapsto xe_1 y$ gives a well defined $M_1$-bilinear isomorphism of $*$-algebras which sends $p$ to $e_1$. Moreover, if $\langle M_1,p\rangle=M_1pM_1$, then $\widetilde{\psi}$ is a trace-preserving isomorphism of von Neumann algebras which fixes $M_1$.
\end{lem}
\begin{proof}
Clearly $\psi$ is surjective and preserves the $*$-algebra structure. Suppose 
$$
\psi\left( \sum\limits_{i=1}^k x_i \otimes y_i\right)=\sum\limits_{i=1}^k x_i f_1 y_i =0.
$$
Then for all $x,y\in M_1$, $\sum\limits_{i=1}^k E_{M_0} (xx_i)E_{M_0}(y_iy)=0$, so by Remark \ref{basisproperties},
{\fontsize{11}{11}{
$$
\sum\limits_{i=1}^k x_i \otimes y_i=\sum\limits_{a\in B} a\otimes a^*\left(\sum\limits_{i=1}^k x_i \otimes y_i\right)\sum\limits_{b\in B} b\otimes b^*=\sum\limits_{a,b\in B}\sum\limits_{i=1}^k a\otimes E_{M_0}(a^*x_i)E_{M_0}(y_i b)b^*=0,
$$}}
where $B=\{b\}$ is a Pimsner-Popa basis for $M_1$ over $M_0$. The last claim follows immediately.
\end{proof}

\begin{defn}
We say an inclusion $M_0\subset M_1\subset (P,\tr_P,p)$ as above is \underline{standard} if it satisfies the conditions of Lemma \ref{basicconstruction}.
\end{defn}

\subsection{The Jones tower and tensor products}\label{prelim}
For the rest of this section, let $M_0\subset (M_1,\tr_1)$ be a strongly Markov inclusion of finite von Neumann algebras, and set $d=[M_1\colon M_0]^{1/2}$. For $n\in\N$, inductively define the basic construction 
$$M_{n+1}=\langle M_n,e_n\rangle=M_n e_n M_n\subset B(L^2(M_n,\tr_n))$$ 
with canonical trace $\tr_{n+1}$ extending $\tr_n$ and satisfying $\tr_{n+1}(xe_n)=d^{-2} \tr_n(x)$ for all $x\in M_n$ where $e_n\in B(L^2(M_n,\tr_n))$ is the Jones projection with range $L^2(M_{n-1},\tr_{n-1})$.
For $n\in\N$, set $E_n=d e_n$ and $v_n=E_nE_{n-1}\cdots E_1$.

\begin{fact} The $E_i$'s satisfy the Temperley-Lieb relations:
\item[(i)] $E_i^2=dE_i=dE_i^*$,
\item[(ii)] $E_iE_j=E_jE_i$ for $|i-j|>1$, and
\item[(iii)] $E_iE_{i\pm 1}E_i=E_i$.
\end{fact}

\begin{prop}
Suppose $N\subset M\subset (P,\tr_P)$, and suppose $A=\{a\}$ is a Pimsner-Popa basis for $P$ over $M$ and $B=\{b\}$ is a Pimsner-Popa basis for $M$ over $N$. Then $AB=\set{ab}{a\in A\text{ and }b\in B}$ is a Pimsner-Popa basis for $P$ over $N$.
\end{prop}
\begin{proof}
For all $x\in P$, 
{\fontsize{10}{10}{
$$
\sum\limits_{ab\in AB} abE_N^P(b^*a^*x)=\sum\limits_{a,b} abE^M_N(E^P_M(b^*a^*x))=\sum\limits_{a,b} abE^M_N(b^*E^P_M(a^*x))=\sum\limits_{a} aE^P_M(a^*x)=x.
$$}}
\end{proof}

\begin{cor}
$M_k\subset (M_n,\tr_n)$ is strongly Markov for all $0\leq k\leq n$.
\end{cor}

Forms of the next lemma appear in \cite{MR1073519},\cite{MR1424954}:
\begin{prop}[Multistep Basic Construction]\label{multistep}
For all $0\leq k\leq n$, let $e_{n-k}^n\in B(L^2(M_{n},\tr_{n}))$ be the projection with range $L^2(M_{n-k},\tr_{n-k})$. There is an isomorphism $\langle M_{n}, e_{n-k}^n\rangle=M_{n} e_{n-k}^n M_{n}\cong M_{n+k}$ which fixes $M_{n}$ and sends $e_{n-k}^n$ to 
$$
f_{n-k}^n=d^{k(k-1)}(e_{n}e_{n-1}\cdots e_{n-k+1})(e_{n+1}e_{n+1}\cdots e_{n-k+2})\cdots(e_{n+k-1}e_{n+k-2}\cdots e_{n}).
$$
Hence the inclusion $M_{n-k}\subset M_{n}\subset (M_{n+k},\tr_{n+k},f^n_{n-k})$ is standard. (See Remark \ref{multisteppicture}).
\end{prop}
\begin{proof}
First, if $B=\{b\}$ is a Pimsner-Popa basis for $M_{n-k+j}$ over $M_{n-k}$, then $\sum_{b\in B}bf_{n-k}^nb^*=f^n_{n-k+j}$, so $M_n f^n_{n-k} M_n=M_{n+k}$ as $f^n_n=1$. This implies $f^n_{n-k}\in M_{n-k}'\cap B(L^2(M_,\tr_n))$ has central support $1$, so $y\mapsto yf^k_{n-k}$ is injective on $M_{n-k}$. Finally, one checks $f^n_{n-k}xf^n_{n-k}=E_{M_{n-k}}(x)f^n_{n-k}$ for all $x\in M_n$ and $E_{M_{n}}(f^n_{n-k})=d^{-2k}$, and the result follows by Lemma \ref{basicconstruction}.
\end{proof}

\begin{rem}\label{global}
Note that $L^2(M_n,\tr_n)$ has left and right actions of $M_0,\dots,M_{2n}$, where $M_i$ acts on the right as $JM_iJ\cong M_i^{\text{op}}$ (we will write $J$ instead of $J_n$). Note that $M_i'=JM_{2n-i}J$, so we define a canonical trace on $M_i'\cap B(L^2(M_n,\tr_n))$ by $\tr_i'(x)=\tr_{2n-i}(Jx^*J)$ for all $x\in M_i'\cap B(L^2(M_n,\tr_n))$. 
\end{rem}

\begin{prop}\label{multiconditional}
The canonical trace-preserving conditional expectation $M_{n+k}\to M_{n+k-i}$ is given by $xf^n_{n-k}y\mapsto d^{-2i}xf^{n}_{n-k+i}y$ where $x,y\in M_n$. The canonical trace-preserving conditional expectation $M_{n-k}'=JM_{n+k}J\to JM_{n+k-i}J=M_{n-k+i}'$ is given by the same formula, only with $x,y\in M_n'=JM_nJ$.
\end{prop}
\begin{proof} We prove the first statement, as the second is similar. By the Markov property, for all $x,y\in M_n$,
$$
\tr_{n+k}(xf^n_{n-k}y)=d^{-2k}\tr_n(xy)=d^{-2i}\tr_{n+k-i}(xf^{n}_{n-k+i}y),
$$
so the map is trace-preserving. Now $M_{n+k-i}$-bilinearity follows from the following two facts:
\item[(i)] for all $1\leq i\leq k$, $M_{n-k}\subset M_{n-k+i}$, so $f^{n}_{n-k+i}f^n_{n-k}=f^{n}_{n-k}$, and
\item[(ii)] $E^{M_{n+k}}_{M_{n+k-i}}(f^{n}_{n-k})=d^{-2i}f^{n}_{n-k+i}$.
\end{proof}

We can now strengthen Proposition 2.7 from \cite{MR1424954}:
\begin{prop}\label{sumb}
The conditional expectation $E_{M_1'}\colon (M_0'\cap B(L^2(M_n,\tr_n)),\tr_0')\to (M_1'\cap B(L^2(M_n,\tr_n)),\tr_1')$ is given by
$$
E_{M_1'}(x)=\frac{1}{d^2}\sum\limits_{b\in B} bxb^*,
$$
where $B$ is a Pimsner-Popa basis for $M_1$ over $M_0$. In particular, the map is independent of choice of basis.
\end{prop}
\begin{proof}
The result follows from Proposition \ref{multiconditional}, since for $x,y\in JM_nJ\subset M_1'$,
$$
\sum\limits_{b\in B} bxf^n_0yb^*=\sum\limits_{b\in B} xbf^n_0b^*y=xf^n_1y.
$$
\end{proof}

Proposition \ref{existence} and a simple induction argument show the following:
\begin{prop}\label{renormalize}
For all $n\in\N$, there are isomorphisms of $M_1-M_1$ bimodules
\begin{align*}
\theta_n\colon \bigotimes\limits^n_{M_0} M_1&\longrightarrow M_n\hsp{by}\\
x_1\otimes\cdots \otimes x_n &\longmapsto x_1 v_1 x_2 v_2\cdots v_{n-1} x_n.
\end{align*}
\end{prop}

\begin{rem}
Recall that $L^2(M_n,\tr_n)$ is the completion of $M_n$ with inner product $\langle x,y\rangle = \tr_n(y^*x)$. As usual, $\theta_n$ gives an isomorphism of Hilbert-bimodules
$$
\bigotimes\limits^n_{M_0} L^2(M_1,\tr_1)\longrightarrow L^2(M_n,\tr_n)
$$
where the tensor product on the left is Connes' relative tensor product with inner product given inductively by
\begin{align*}
\langle x_1\otimes u, y_1\otimes v\rangle_n &=\langle E_{M_0}(y_1^* x_1) u,v\rangle_{n-1}\\
\langle u\otimes x_n, v\otimes y_n\rangle_n &=\langle u,vE_{M_0}(y_n x_n^*)\rangle_{n-1}.
\end{align*}
\end{rem}

\begin{defn} Given $x\in M_1$, we get
\item[(1)] left and right multiplication operators
$$
L(x),R(x)\colon \bigotimes\limits^{n}_{M_0} L^2(M_1,\tr_1) \longrightarrow \bigotimes\limits^{n}_{M_0} L^2(M_1,\tr_1)
$$
by $L(x)(v)=xv$ and $R(x)(v)=vx$, and
\item[(2)] left and right creation operators 
$$
L_x,R_x\colon \bigotimes\limits^n_{M_0} L^2(M_1,\tr_1) \longrightarrow \bigotimes\limits^{n+1}_{M_0} L^2(M_1,\tr_1)
$$
by $L_x(v)=x\otimes v$ and $R_x(v)=v\otimes x$. 
\end{defn}

\begin{fact}
For $x\in M_1$, we have
\begin{align*}
L_x^*(y_1\otimes\cdots \otimes y_{n+1}) &= E_{M_0}(x^* y_1) y_2\otimes\cdots \otimes y_{n+1}\hsp{and}\\
R_x^*(y_1\otimes\cdots \otimes y_{n+1}) &=y_1\otimes \cdots\otimes y_{n} E_{M_0}(y_{n+1} x^*).
\end{align*}
\end{fact}

\begin{lem}\label{invariant}
If $A$ is a $\C$-algebra, $V_1$ is a right $A$-module, $V_2$ is an $A-A$ bimodule, and $V_3$ is a left $A$-module, then for each $A$-invariant $v_2\in V_2$, the map
$$
v_1\otimes v_3 \longmapsto v_1\otimes v_2\otimes v_3
$$
defines a linear map $\phi_{v_2}\colon V_1\otimes_A V_3 \to V_1\otimes_A V_2\otimes_A V_3$. Moreover, the map $v\mapsto\phi_v$ on $A'\cap V_2=\set{v\in V_2}{av=va\hsp{for all}a\in A}$ is $\C$-linear.
\end{lem}
\begin{proof}
Middle $A$-linearity is satisfied as $v_2$ is $A$-invariant.
\end{proof}

\begin{rem}
This lemma gives an alternate proof that the map $E_{M_1'}$ is well defined in Proposition \ref{sumb}. By Remark \ref{basisproperties}, $d^{-2}\sum_{b\in B} b\otimes b^*$ is independent of the choice of Pimsner-Popa basis $B$, so the composite map
$$
x\longmapsto\phi_x\longmapsto \phi_x\left(d^{-2}\sum\limits_{b\in B} b\otimes b^*\right)=d^{-2}\sum_{b\in B} b\otimes x\otimes b^*\longmapsto d^{-2}\sum_{b\in B} bx b^*
$$
on $M_0'\cap B(L^2(M_n,\tr_n))$ is independent of the choice. Moreover, the result is $M_1$-invariant, since for any unitary $u\in M_1$, $\set{ub}{b\in B}$ is another Pimsner-Popa basis for $M_1$ over $M_0$.
\end{rem}

\subsection{Definition of the planar algebra}\label{padefn}
We refer the reader to \cite{math/9909027}, \cite{MR1865703}, \cite{0902.1294}, or \cite{1007.1158v2} for the definition of a planar $*$-algebra. To define our planar algebra, we need the following:

In \cite{math/9909027}, it was shown how to endow the tower of relative commutants of an extremal, finite index $II_1$-subfactor with the structure of a spherical subfactor planar algebra. In fact, the only essential ingredient is a Pimsner-Popa basis, so the same construction gives a planar $*$-algebra structure on the vector spaces $P_{n,\pm}$ ($n\geq 0$) given by $P_{n,+}=\theta_{n}^{-1}(M_0'\cap M_n)$ and $P_{n,-}=\theta_n^{-1}(M_1'\cap M_{n+1})$, except that we now allow $*$'s in shaded regions, and we no longer require extremality (the resulting planar algebra need not be spherical).

Suppose a planar $(k,\pm)$-tangle $T$ is arranged so that 
\be
\item[(1)] all the input and output disks are horizontal rectangles with all strings (that are not closed loops) emanating from the top edges of the rectangles,
\item[(2)] all the input disks are in disjoint horizontal bands and all maxima and minima of strings are at different vertical levels, and not in the horizonal bands defined by the input disks.
\item[(3)] the distinguished (starred) intervals of all the disks are at the left edges of the rectangles. (In the sequel, we will assume this convention and omit the $*$'s.)
\ee
Suppose also that $T$ has $s$ input rectangles, and input rectangle $j$ has $2r_i$ strings emanating from the top. We define the action of $T$ on an $s$-tuple $u=(u_1, \cdots, u_s)$ where $u_j\in P_{r_j,\pm_j}$ and $\pm_j=\pm$ if the region just below input rectangle $j$ is unshaded or shaded respectively. 

We read the action of $T$ on $u$ by sliding a horizontal line through the tangle from bottom to top. For a fixed vertical $y$-value, off the input disks' horizontal bands and away from the relative extrema of the strings, the horizontal line will meet $n_y$ shaded regions from left to right. One should think of the shaded regions along this line as elements of $M_1$ and the unshaded regions as the symbols $\otimes_{M_0}$. Near the top, the line will meet $k$ or $k+1$ shaded regions depending on whether the left-most region of $T$ is unshaded or shaded respectively. We illustrate a typical $(3,+)$-tangle with the horizontal line about half way through its travel:
$$\input{pictures/exampletangle}$$
For each $y$ coordinate of the horizontal line, one reads off an $M_i$-invariant element $\eta_y\in\bigotimes^{n_y}_{M_0} M_1$, where $i=0$ if $T$ is a $(k,+)$-tangle and $i=1$ if $T$ is a $(k,-)$-tangle. 

The element  $\eta_y$ begins as $1\in M_i$ near the bottom, and it remains constant as long as the horizontal line meets neither maxima, minima, nor rectangles. If the horizontal line passes input rectangle $j$ for which exactly $t$ shaded regions sit to the left, then we insert $u_j$ into $\eta_y$ by applying Lemma \ref{invariant} with $v_2=u_j$, 
$$V_1=\bigotimes^t_{M_0} M_1,\hs V_2=P_{r_j,\pm_j},\hsp{and} V_1=\bigotimes^{n_y-t}_{M_0} M_1.$$
Note that $V_1,V_3$ are considered as $M_l$-modules and $P_{r_j,\pm_j}$ is an $M_l-M_l$ bimodule, where $l=0$ if $\pm_j=+$ and $l=1$ if $\pm_j=-$. Note that inserting $u_j$ into $\eta_y$ gives an $M_i$-invariant vector.

As the horizontal line passes a maximum or minimum, $\eta_y$ changes according to Figure \ref{readtangle} where the changes indicated on the tensors are to be inserted into the position indicated by the shaded regions on the horizontal (dashed) line.
\begin{figure}[!ht]
\begin{align*}
\input{pictures/shadedmin}&&x\otimes y&\longmapsto x\otimes 1\otimes y\\
\input{pictures/unshadedmin}&&x&\longmapsto d^{-1}\sum\limits_{b\in B} xb\otimes b^*=d^{-1}\sum\limits_{b\in B} b\otimes b^*x\\
\input{pictures/shadedmax}&&x\otimes y\otimes z&\longmapsto dxE_{M_0}(y)\otimes z=dx\otimes E_{M_0}(y)z\\
\input{pictures/unshadedmax}&&x\otimes y&\longmapsto xy.
\end{align*}
\caption{Reading planar tangles in standard form}\label{readtangle}
\end{figure}
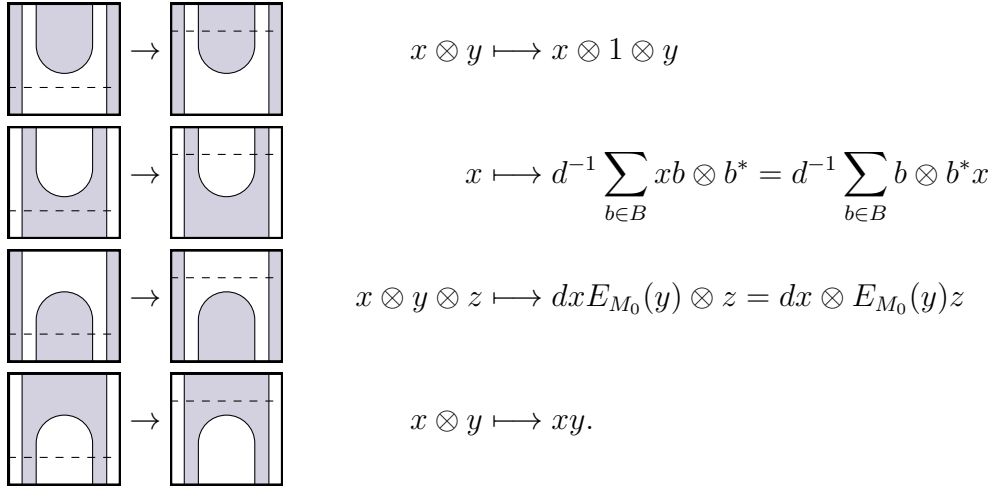
With the exception of one case, each of these maps is an $M_1-M_1$ bimodule map, so it will preserve $M_i$-invariant elements. The remaining case to consider is when the left-most or right-most shaded region is capped off by applying the third map pictured above, which is an $M_0-M_0$ bimodule map. But this will only occur when the distinguished (starred) interval of the external disk meets an unshaded region, so $i$ would have to be $0$ from the beginning.

The action of the tangle on $u$ is the element $\eta_y\in P_{k,\pm}$ read for horizontal lines sufficiently close to the top. The $*$-structure is the same from that of \cite{math/9909027}.

\subsection{Burns' treatment of the rotation operator on $P_{n,+}$}
The key to showing that the $P_{n,\pm}$'s define a planar algebra is the isotopy invariance which comes from the existence of the rotation on $P_{n,\pm}$. A particularly elegant treatment of this is due to Michael Burns, but it only appears in his thesis \cite{burns}, so we include a proof below for the convenience of the reader.

\begin{defn}\label{rotationdef}
Let $B$ be a Pimsner-Popa basis of $M_1$ over $M_0$. For 
$$x=x_1\otimes\cdots \otimes x_n \in \bigotimes^n_{M_0} M_1,$$
define $\D\rho(x) = \sum\limits_{b\in B} L_b R_b^*(x)= \sum\limits_{b\in B} b\otimes x_1\otimes \cdots\otimes x_{n-1} E_{M_0}(x_n b^*)$.
\end{defn}

\begin{prop}\label{rotationwelldefined}
The map $\rho$ preserves $P_{n,+}$, and its restriction to $P_{n,+}$ is independent of the choice of $B$.
\end{prop}
\begin{proof}
Middle  linearity is respected by $\rho$, so it is well defined, though it will depend on $B$. By Lemma \ref{invariant} and Remark \ref{basisproperties}, for $M_0$-invariant $x$, the sum 
$$
\sum\limits_{b\in B} b\otimes x\otimes b^*
$$
is independent of $B$. We obtain $\rho$ by applying an $M_0-M_0$ bilinear map which does not involve $B$, so the restriction of $\rho$ is $M_0$-invariant and independent of $B$.
\end{proof}

\begin{thm}[Burns]\label{rotation}
For $x\in P_{n,+}$,
$$
\langle \rho(x), y_1\otimes\cdots \otimes y_n\rangle = \langle x, y_2\otimes \cdots \otimes y_{n}\otimes y_1\rangle.
$$
\end{thm}
\begin{proof}
As $\D \rho(x)=\sum\limits_{b\in B} L_b R_b^*(x)$, we have
{\fontsize{10}{10}{
\begin{align*}
\langle \rho(x), y_1\otimes\cdots \otimes y_n\rangle &=\sum\limits_{b\in B}\langle  L_b R_b^*  x, y_1\otimes\cdots \otimes y_n\rangle
=\sum\limits_{b\in B}\langle  x,R_b L_b^* y_1\otimes\cdots \otimes y_n\rangle\\
&=\sum\limits_{b\in B}\langle  x, E_{M_0}(b^*y_1)y _2\otimes\cdots \otimes y_n\otimes b\rangle
=\sum\limits_{b\in B}\langle  E_{M_0}(b^*y_1)^*x, y _2\otimes\cdots \otimes y_n\otimes b\rangle\\
&=\sum\limits_{b\in B}\langle  xE_{M_0}(b^*y_1)^*, y _2\otimes\cdots \otimes y_n\otimes b\rangle
=\sum\limits_{b\in B}\langle  x, y _2\otimes\cdots \otimes y_n\otimes bE_{M_0}(b^*y_1)\rangle\\
&= \langle x, y_2\otimes \cdots \otimes y_{n}\otimes y_1\rangle.
\end{align*}}}
\end{proof}

\begin{cor}
$\rho^n=\id$ on $P_{n,+}$.
\end{cor}

\begin{cor}
The rotation $\rotation{}$ on $P_{n,+}$ is well defined.
\end{cor}

\subsection{The rotation on $P_{n,-}$} We mimic Burns' treatment of the rotation on $P_{n,+}$ to define the rotation on $P_{n,-}$.

\begin{defn}\label{rotationminusdef}
Let $B$ be a Pimsner-Popa basis of $M_1$ over $M_0$. For 
$$x=x_1\otimes\cdots \otimes x_{n+1} \in \bigotimes^{n+1}_{M_0} M_1,$$
define $\D\sigma(x) = \sum\limits_{b\in B} R(b^*)R_1^* L_b(x)= \sum\limits_{b\in B} b\otimes x_1\otimes \cdots\otimes x_{n} E_{M_0}(x_{n+1})b^*$.
\end{defn}

\begin{prop}
The map $\sigma$ preserves $P_{n,-}$, and its restriction to $P_{n,-}$ is independent of the choice of $B$.
\end{prop}
\begin{proof}
Similar to Proposition \ref{rotationwelldefined}.
\end{proof}

\begin{thm}\label{minusrotation}
For $x\in P_{n,-}$,
$$
\langle \sigma(x), y_1\otimes\cdots \otimes y_{n+1}\rangle = \langle x, y_2\otimes \cdots \otimes y_{n}\otimes y_{n+1}y_1\otimes 1\rangle.
$$
\end{thm}
\begin{proof}
Similar to Theorem \ref{rotation}.
\end{proof}

\begin{cor}
$\sigma^n=\id$ on $P_{n,-}$.
\end{cor}
\begin{proof}
As $\sigma$ preserves $P_{n,-}$, we repeatedly apply Theorem \ref{minusrotation} for $x\in P_{n,-}$ to get\begin{align*}
\langle \sigma^n (x), y_1\otimes\cdots \otimes y_{n+1}\rangle &=\langle \sigma^{n-1}(x), y_2\otimes \cdots \otimes y_{n}\otimes y_{n+1}y_1\otimes 1\rangle\\
&=\langle \sigma^{n-2}(x), y_3\otimes \cdots \otimes y_{n}\otimes y_{n+1}y_1\otimes y_2\otimes 1\rangle\\
&=\cdots=\langle x, y_{n+1}y_1\otimes y_2\otimes \cdots \otimes y_{n}\otimes 1\rangle.
\end{align*}
We then invoke Burns' trick again to get
\begin{align*}
\langle x, y_{n+1}y_1\otimes y_2\otimes \cdots \otimes y_{n}\otimes 1\rangle&=\langle y_{n+1}^*x, y_1\otimes \cdots \otimes y_{n}\otimes  1\rangle\\
&=\langle x y_{n+1}^*,y_1\otimes \cdots \otimes y_{n}\otimes 1\rangle\\
&=\langle x, y_1\otimes \cdots \otimes y_{n}\otimes  y_{n+1}\rangle.
\end{align*}
\end{proof}

\begin{cor}
The rotation $\minusrotation{}$ on $P_{n,-}$ is well defined.
\end{cor}

\subsection{Uniqueness of the canonical planar algebra}
We have the following facts whose proofs are similar to those in \cite{math/9909027} and will be omitted. We shade tangles as much as possible, but sometimes we will not have enough information.

\begin{prop}[Multiplication]\label{multiplication1}
Suppose $x,y\in M_n$ such that 
$$\theta_n^{-1}(x)=x_1\otimes\cdots\otimes x_n\hsp{and} \theta_n^{-1}(y)=y_1\otimes\cdots\otimes y_n$$
Then
$$
\theta_n^{-1}(xy)=
{\fontsize{10}{10}{
\begin{cases}
x_1\otimes\cdots \otimes x_k E_{M_0}(x_{k+1}E_{M_0}(x_{k+2}(\cdots) y_{k-1})y_{k})\otimes y_{k+1}\otimes\cdots\otimes y_{2k} &n=2k\\
x_1\otimes\cdots \otimes x_{k+1} E_{M_0}(x_{k+2}E_{M_0}(x_{k+3}(\cdots) y_{k-1})y_{k}) y_{k+1}\otimes\cdots\otimes y_{2k+1}&n=2k+1.
\end{cases}}}
$$
\end{prop}

\begin{rem}\label{multiplication2}
If $x,y$ as above are in $M_i'\cap M_n$ where $i\in\{0,1\}$, then\\
$\D
\theta_n^{-1}(xy)=\multiplication{\text{{\fontsize{10}{10}{$x_1\otimes\cdots\otimes x_n$}}}}{\text{{\fontsize{10}{10}{$y_1\otimes\cdots\otimes y_n$}}}}
$
where the shading depends on $i$ and the parity of $n$.
\end{rem}

\begin{prop}[$*$-Structure]\label{starstructure}
Suppose $x\in M_n$ such that $\theta_n^{-1}(x)=x_1\otimes\cdots\otimes x_n$. Then $\theta_n^{-1}(x^*)=x_n^*\otimes\cdots \otimes x_1^*$. 
\end{prop}

\begin{prop}[Jones Projections]
\item[(1)] For $n\geq 1$, the Jones projection $E_n\in P_{n+1,+}$ is given by $\jonesprojection$\,.
\item[(2)] For $n\geq 2$, the Jones projection $E_{n}\in P_{n,-}$ is given by $\minusjonesprojection$\,.
\end{prop}

\begin{rem}\label{multisteppicture}
The multistep basic construction projection of Proposition \ref{multistep} is given by
$f^n_{n-k}=d^{-k}\multistep{n-k}{k}\,$. 
\end{rem}

\begin{prop}[Inclusions]
\item[(1)]
Let  $i_n\colon M_0'\cap M_n\to M_0'\cap M_{n+1}$ be the inclusion. Then the inclusion\\ 
$\theta_{n+1}\circ i_n\circ \theta_n\colon P_{n,\pm}\to P_{n+1,\pm}$ is given by $\inclusion{}$\,.
\item[(2)] If $x\in P_{n,-}$, then $\D\minusinclusion{x}=x\in P_{n+1,+}$.
\end{prop}

\begin{prop}[Conditional Expectations]
\item[(1)]
The conditional expectation $\theta_{n-1}^{-1}\circ E_{M_{n-1}}\circ \theta_n\colon P_{n,+}\to P_{n-1,+}$ is given by $$d^{-1}\cdot \conditional{}\,.$$
\item[(2)]
The conditional expectation $\theta_{n}^{-1}\circ E_{M_1'}\circ \theta_n\colon P_{n,+}\to P_{n-1,-}$ (see Proposition \ref{sumb}) is given by 
$$d^{-1}\cdot \commutant{}\,.$$
\end{prop}

\begin{nota} We use the notation from \cite{0912.1320}:
\item[(1)] Denote the annular capping maps $P_{n,+}\to P_{n-1,+}$ by $\alpha_j$ as shown:
$$
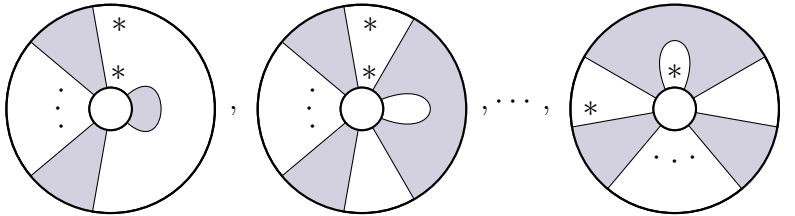
\input{pictures/alphas}\hs,
$$
i.e., the $i\thh$ and $(i+1)\thh$ (modulo $2n$) internal boundary points are joined by a string and all other internal boundary points are connected to external boundary points such that
\be
\item[(i)] If $i=1$, then the first external point is connected to the third internal point.
\item[(ii)] If $1<i<2n$, then the first external point is connected to the first internal point.
\item[(iii)] If $i=2n$, then the first external point is connected to the $(2n-1)\thh$ internal point.
\ee
\item[(2)] Denote the annular cupping maps $P_{n-1,+}\to P_{n,+}$ by $\beta_j$ as shown:
$$
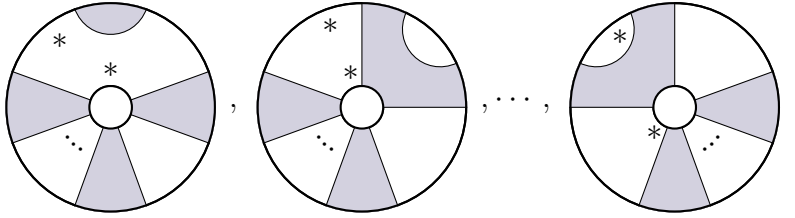
\input{pictures/betas}\hs,
$$
i.e., $\beta_j$ is $\alpha_j$ turned inside out. 
\end{nota}

The following lemma is similar to a result in \cite{MR2047470}:
\begin{lem}\label{technical}
Suppose $P_\bullet$ is a planar $*$-algebra with modulus $d\neq 0$ and $Q_{n,\pm}\subset P_{n,\pm}$ are $*$-subalgebras which are closed under the following operations:
\item[(1)] left and right multiplication by tangles $E_n=\D \jonesprojection\in P_{n+1,+}$ for $n\in\N$;
\item[(2)] The maps from $P_{n,+}$ as follows:
\begin{align*}
\alpha_n&=\conditional{}\,\colon P_{n,+}\to P_{n-1,+},
\hspace{.15in}\beta_{n+1}=\inclusion{}\,\colon P_{n,+}\to P_{n+1,+},\\
\gamma_n^+&=\commutant{}\,\colon P_{n,+}\to P_{n-1,-};\hsp{and}
\end{align*}
\item[(3)] the map $i_n^-=\minusinclusion{}\,\colon P_{n,-}\to P_{n+1,+}$.
\item
Then the $Q_{n,\pm}$ define a planar $*$-subalgebra $Q_\bullet\subset P_\bullet$.
\end{lem}
\begin{proof} As $Q_{n,\pm}$ is closed under multiplication and $*$, it suffices to show $Q_\bullet$ is closed under all annular maps. To show this, it suffices to show all $\alpha_j$'s, all $\beta_j$'s, and both rotations by $1$ preserve $Q_\bullet$. 

First, note that the maps $\gamma^-_n\colon P_{n,-}\to P_{n-1,+}$ and $i_n^+\colon P_{n,+}\to P_{n+1,-}$  given by
\begin{align*}
\gamma^-_n(x)&=\minuscommutant{x}=\frac{1}{d} \alpha_{n+2}(E_n E_{n-1}\cdots E_1 \cdot \beta_{n+2}(i^-_nx))\cdot E_1 E_2\cdots E_n)\hsp{and}\\
i_n^+(x)&=\commutantinclusion{x}=\gamma_{n+2}^+( (E_1 E_2\cdots E_n) \cdot \beta_{n+2}\beta_{n+1}(x) \cdot (E_{n+1} E_n\cdots E_1))
\end{align*}
preserve $Q_\bullet$. 

We show all $\alpha_j$'s preserve $Q_\bullet$. For $j<n$ and $x\in Q_n$, 
$$
\alpha_j(x) = \frac{1}{d} \alpha_{n}\alpha_{n+1}( (E_n E_{n-1}\cdots E_j)\cdot \beta_{n+1}(x) \cdot (E_{n}))).
$$
The case $n<j<2n$ is similar. It is clear $\alpha_{2n}(x)=\alpha_{2n-1}(i^-_{n-1}(\gamma^+_n(x)))$. 

We show all $\beta_j$'s preserve $Q_\bullet$. If $j<n+1$, we have
$$
\beta_j(x) = (E_{j} E_{j-1}\cdots E_n)\cdot \beta_{n+1}(x).
$$
The case $n+1<j<2n+2$ is similar. It is clear $\beta_{2n+2}(x)=\alpha_{2}\gamma^-_{n+1}\gamma^+_n(x)$. 

We show both rotations by $1$ preserve $Q_\bullet$. We have
\begin{align*}
\rotationbyone{x}&=\frac{1}{d}\gamma^+_{n+1}\alpha_{2n+2} i^-_{n+1} i^+_{n} \alpha_n \beta_{n+1}(x)\hsp{and}\\
\minusrotationbyone{x}&=\alpha_{n+1}\beta_{n+2}\alpha_{2n+1} i^-_n(x).
\end{align*}
\end{proof}

\begin{thm}\label{canonical}
Given a strongly Markov inclusion $M_0\subset (M_1,\tr_1)$, there is a unique planar $*$-algebra $P_\bullet$ of modulus $d=[M_1\colon M_0]^{1/2}$ where 
$$
P_{n,+}=\theta^{-1}_n(M_0'\cap M_{n})\hsp{and}P_{n,-}=\theta^{-1}_{n+1}(M_1'\cap M_{n+1})
$$
such that the multiplication is given by Remark \ref{multiplication2},
\item[(0)] for all tangles $T$ with $n$ input disks, $T(\xi_1^*\otimes\cdots \otimes \xi_n^*)=T^*(\xi_1\otimes\cdots\otimes \xi_n)^*$ where  for $\xi_i\in P_{n_i,\pm_i}$, $\xi^*_i$ is as in Proposition \ref{starstructure} and $T^*$ is the mirror image of $T$;
\item[(1)] for $n\in\N$, 
$\D E_n=\jonesprojection\in P_{n+1,+}$;
\item[(2)]
for $x\in P_{n,+}$ and $B$ a Pimsner-Popa basis for $M_1$ over $M_0$,
\begin{align*} 
\conditional{x}&=dE_{M_{n-1}}(x),\hspace{1in}\inclusion{x}=x\in P_{n+1,+},\hsp{and}\\
\commutant{x}&=dE_{M'_1}(x)=d^{-1}\sum\limits_{b\in B} bxb^*;\hsp{and}
\end{align*}
\item[(3)] for $x\in P_{n,-}$, $\D\minusinclusion{x}=x\in P_{n+1,+}$.
\end{thm}
\begin{proof}
Uniqueness follows from Lemma \ref{technical}. Existence follows from the existence of the canonical relative commutant planar $*$-algebra.
\end{proof}

\section{The planar algebra isomorphism for finite dimensional C$^*$-algebras}\label{sec:iso}

We now restrict our attention to a connected unital inclusion $M_0\subset M_1$ of finite dimensional C$^*$-algebras with the Markov trace. We will show that in this case, the canonical planar algebra of Theorem \ref{canonical} is isomorphic to the bipartite graph planar algebra \cite{MR1865703} of the Bratteli diagram.

\subsection{Loop algebras}\label{loop}
We define loop algebras in the spirit of \cite{MR1865703} which are another description of Ocneanu and Sunder's path algebras \cite{MR999799},\cite{MR1473221} with a more GNS (rather than spatial) flavor. 

\begin{nota} For this section, let $\Gamma$ be a finite, connected, bipartite multi-graph. Let $\V_\pm$ denote the set of even/odd vertices of $\Gamma$, and let $\E$ denote the edge set of $\Gamma$. Usually we will denote edges by $\varepsilon$ and $\xi$. All edges will be directed from even to odd vertices, so we have source and target functions $s\colon \E\to \V_+$ and $t\colon \E\to \V_-$. We will write $\varepsilon^*$ to denote an edge $\varepsilon$ traversed from an odd vertex to an even vertex, and we define source and target functions $s\colon \E^*=\set{\varepsilon^*}{\varepsilon\in\E}\to \V_-$ and $t\colon \E^*\to \V_+$ by $s(\varepsilon^*)=t(\varepsilon)$ and $t(\varepsilon^*)=s(\varepsilon)$. Let $m_+\colon \V_+\to \N$ be a dimension (row) vector for the even vertices. For $v\in \V_-$, define the dimension (row) vector for the odd vertices by
$$
m_-(v)=\sum\limits_{t(\varepsilon)=v} m_+(s(\varepsilon)).
$$
Let $\Lambda$ be the bipartite adjacency matrix for $\Gamma$ ($\Lambda_{i,j}$ is the number of times the $i\thh$ vertex in $\V_+$ is connected to the $j\thh$ vertex in $\V_-$).
\end{nota}

\begin{rem}
Given $(\Gamma,m_+)$, we can associate a connected unital inclusion of finite dimensional C$^*$-algebras $M_0\subset M_1$. We set
$$
M_0=\bigoplus\limits_{v\in\V_+} M_{m_+(v)}(\C)\hsp{and} M_1=\bigoplus\limits_{v\in \V_-} M_{m_-}(\C),
$$
and the inclusion is such that $\Gamma$ is the Bratteli diagram for the inclusion, and $\Lambda$ is the inclusion matrix ($\Lambda_{i,j}$ is the number of times the $i\thh$ simple summand of $M_0$ is contained in the $j\thh$ simple summand of $M_1$). Conversely, given such an inclusion, we get a finite, connected, bipartite multi-graph (the Bratteli diagram) and a dimension vector $m_+$ (corresponding to the simple summands of $M_0$). 
\end{rem}

\begin{defn}
Let $G_{0,\pm}$ be the complex vector space with basis $\V_\pm$ respectively. For $n\in\N$, $G_{n,\pm}$ will denote the complex vector space with basis loops of length $2n$ on $\Gamma$ based at a vertex in $\V_\pm$ respectively. 
\end{defn}

We discuss the vector spaces $G_{n,+}$. The spaces $G_{n,-}$ are similar, and it is clear what the corresponding notation should be and how they will behave.

\begin{nota}\label{paths}
Loops in $G_{n,+}$ will be denoted $[\varepsilon_1\varepsilon_2^*\cdots \varepsilon_{2n-1}\varepsilon_{2n}^*]$. Any time we write such a loop, it is implied that 
\be
\item[(i)]
$t(\varepsilon_i)=s(\varepsilon_{i+1}^*)=t(\varepsilon_{i+1})$ for all odd $i<2n$,
\item[(ii)]
$t(\varepsilon_{i}^*)=s(\varepsilon_i)=s(\varepsilon_{i+1})$ for all even $i<2n$, and
\item[(iii)]
$t(\varepsilon_{2n}^*)=s(\varepsilon_{2n})=s(\varepsilon_1)$.
\ee
For a loop $\ell= [\varepsilon_1\varepsilon_2^*\cdots \varepsilon_{2n-1}\varepsilon_{2n}^*]\in G_{n,+}$, we define the following paths in $\ell$:
\begin{align*}
\ell_{[1,k]}&=
\begin{cases}
\varepsilon_1\varepsilon_2^*\cdots \varepsilon_{k-1}\varepsilon_k^* &\hsp{$k$ even}\\
\varepsilon_1\varepsilon_2^*\cdots \varepsilon_{k-1}^*\varepsilon_k &\hsp{$k$ odd}
\end{cases}\\
\ell_{[k,2n]}&=
\begin{cases}
\varepsilon_{k}\varepsilon_{k+1}^*\cdots \varepsilon_{2n-1}\varepsilon_{2n}^* &\hsp{$k$ odd}\\
\varepsilon_{k}^*\varepsilon_{k+1}\cdots \varepsilon_{2n-1}\varepsilon_{2n}^* &\hsp{$k$ even}
\end{cases}
\end{align*}
where $1\leq k\leq 2n$. Similarly, we may define the path $\gamma_{[j,k]}(\ell)$ in $\ell$ to be the $j\thh-k\thh$ entries of $\ell$ for $1\leq j\leq k\leq 2n$.
\end{nota}

\begin{defn}\label{adjoint}
Define an antilinear map $*$ on $G_{n,+}$ by the antilinear extension of the map
$$
[\varepsilon_1\varepsilon_2^*\cdots \varepsilon_{2n-1}\varepsilon_{2n}^*]^*=[\varepsilon_{2n}\varepsilon_{2n-1}^*\cdots \varepsilon_{2}\varepsilon_{1}^*].
$$
There is also an obvious notion of taking $*$ of a path $\gamma_{[k,l]}(\ell)$ for a loop $\ell\in G_{n,+}$. We define a multiplication on $G_{n,+}$ by
$$
\ell_1\cdot \ell_2 = \delta_{(\ell_1)_{[n+1,2n]}^*,(\ell_2)_{[1,n]}} [(\ell_1)_{[1,n]}(\ell_2)_{[n+1,2n]}].
$$
It is clear that $*$ is an involution for $G_{n,+}$ under this multiplication.
\end{defn}

\begin{rem}
We can think of a loop in $G_{n,+}$ as a path up and down the multi-graph $\Gamma_n$ corresponding to the Bratteli diagram for the inclusions
$$
M_0\subset M_1\subset \cdots \subset M_{n},
$$
which is obtained by reflecting $\Gamma$ a total of $n-1$ times, as the inclusion matrix of $M_j\subset M_{j+1}$ is given by $\Lambda$ or $\Lambda^T$ if $j$ is even or odd, respectively \cite{MR696688}.
\end{rem}

\begin{defn}
Let $\widetilde{\Gamma}$ be the augmentation of the bipartite graph $\Gamma$ by adding a distinguished vertex $\star$ which is connected to each $v\in \V_+$ by $m_+(v)$ distinct edges. These edges are oriented so they begin at $\star$. We will denote these added edges by $\eta's$ (and $\zeta$'s and $\kappa$'s when necessary). 
\end{defn}

\begin{defn}
For $n\in\Z_{\geq 0}$, let $A_n$ be the algebra defined as follows: a basis of $A_n$ will consist of loops of length $2n+2$ on $\widetilde{\Gamma}$ of the form
$$
[\eta_1 \varepsilon_1\varepsilon_2^*\cdots \varepsilon_{2n-1}\varepsilon_{2n}^* \eta_2^*]
$$
i.e., the loops start and end at $\star$, but remain in $\Gamma$ otherwise. Note that we have an obvious $*$-structure on each $A_n$. Multiplication will be given as follows: if one defines the similar path notation as in Notation \ref{paths}, then we have
$$
\ell_1\cdot \ell_2 = \delta_{(\ell_1)_{[n+2,2n+2]}^*,(\ell_2)_{[1,n+1]}} [(\ell_1)_{[1,n+1]}(\ell_2)_{[n+2,2n+2]}].
$$
\end{defn}

\begin{rem}
We can think of a loop in $A_n$ as a path up and down the multi-graph $\widetilde{\Gamma}_n$ corresponding to the Bratteli diagram for the inclusions
$$
\C\subset M_0\subset M_1\subset \cdots \subset M_{n}.
$$
\end{rem}

\begin{defn}[Inclusions]\label{inclusions}
The inclusion $A_n\to A_{n+1}$ is given by the linear extension of
$$
[\eta_1 \varepsilon_1\varepsilon_2^*\cdots \varepsilon_{2n-1}\varepsilon_{2n}^* \eta_2^*]\longmapsto 
\begin{cases}\D
\sum\limits_{s(\varepsilon)=s(\varepsilon_n)} [\eta_1 \varepsilon_1\varepsilon_2^*\cdots\varepsilon_n^*\varepsilon\varepsilon^*\varepsilon_{n+1} \cdots\varepsilon_{2n-1}\varepsilon_{2n}^* \eta_2^*] & \hsp{$n$ even} \\ \D
\sum\limits_{s(\varepsilon)=t(\varepsilon_n)} [\eta_1 \varepsilon_1\varepsilon_2^*\cdots \varepsilon_n\varepsilon^*\varepsilon\varepsilon_{n+1}^*\cdots\varepsilon_{2n-1}\varepsilon_{2n}^* \eta_2^*] & \hsp{$n$ odd.} 
\end{cases}
$$
We identify $A_n$ with its image in $A_{n+1}$.
\end{defn}

\begin{rem}
The inclusion identifications allow us to define a multiplication $A_m\times A_n\to A_{\max\{m,n\}}$ by including $A_m,A_n$ into $A_{\max\{m,n\}}$ and using the regular multiplication. Explicitly, if $\ell_1\in A_m$ and $\ell_2\in A_n$ with $m\leq n$, then
$$
\ell_1\cdot \ell_2 = \delta_{(\ell_1)_{[m+2,2m+2]},(\ell_2)_{[1,m+1]}} [(\ell_1)_{[1,m+1]} (\ell_2)_{[m+2,2n+2]}].
$$
The case $m\geq n$ is similar. 
\end{rem}

\subsection{Towers of loop algebras}\label{towerofloops}
We provide an isomorphism of the tower $(M_n)_{n\geq 0}$ coming from a connected unital inclusion of finite dimensional C$^*$-algebras with the Markov trace and the corresponding tower $(A_n)_{n\geq 0}$ of loop algebras. Assume the notation of Subsection \ref{loop}.

\begin{prop}
For $n\geq 0$, if $S_i$ is the $i\thh$ simple summand of of $M_n$, then loops $\ell$ in $A_n$ for which $\ell_{[1,n+1]}$ ends at the corresponding vertex of $\widetilde{\Gamma}_n$ form a system of matrix units for a simple algebra isomorphic to $S_i$.
\end{prop}
\begin{cor}\label{iso}
For $n\in\Z_{\geq 0}$, there is an algebra $*$-isomorphism $A_n\cong M_n$.
\end{cor}

\begin{rem} At this point, we only choose such isomorphisms $\varphi\colon A_n\to M_n$ for $n=0,1$ which respects the inclusion given in \ref{inclusions}. The key step will be to inductively define isomorphisms so we may identify the Jones projections.
\end{rem}

\begin{defn}
Following \cite{MR696688}, let $\lambda_i$ be the Markov trace (column) vector for $M_i$ for $i=0,1$ such that 
$$m_+ \lambda_0=1=m_- \lambda_1,$$ 
so $\lambda_i$ gives the traces of minimal projections in the simple summands of $M_i$ for $i=0,1$. In order for the trace on $M_1$ to restrict to the trace on $M_0$, we must have $\Lambda \lambda_1 = \lambda_0$.

Recall that the inclusion matrix for $M_n\subset M_{n+1}$ is given by $\Lambda$ if $n$ is even and $\Lambda^T$ if $n$ is odd. This means that to extend the trace, we must have $\Lambda\Lambda^T\lambda_0=d^{-2}\lambda_0$, $\Lambda^T\Lambda\lambda_1=d^{-2}\lambda_1$, and $\lambda_n =d^{-2} \lambda_{n-2}$ for all $n\geq 2$,  where  $\lambda_n$ is the Markov trace vector for $M_n$ and $d=\sqrt{\|\Lambda^T\Lambda\|}=\sqrt{\|\Lambda\Lambda^T\|}$. 
\end{defn}

\begin{defn}\label{fp}
Let $\D\lambda=
\begin{pmatrix}
\lambda_0\\
d \lambda_i
\end{pmatrix}$, a Frobenius-Perron eigenvector for $\D
\begin{pmatrix}
0 & \Lambda\\
\Lambda^T & 0
\end{pmatrix}$.
\end{defn}

\begin{defn}[Traces]\label{trace}
We define a trace on $A_0$ by
$$
\tr_0([\eta_1\eta_2^*])=\begin{cases}
\lambda(t(\eta_1))=\lambda_0(t(\eta_1)) &\hsp{if $\eta_1=\eta_2$}\\
0 &\hsp{else.}
\end{cases}
$$
Suppose $\ell=[\eta_1\varepsilon_1\varepsilon_2^*\cdots \varepsilon_{2n-1}\varepsilon_{2n}^*\eta_2^*]\in A_n$ with $n\geq 1$. We define a trace on $A_n$ by
$$
\tr_{n}(\ell) = \begin{cases}
d^{-n}\lambda(s(\varepsilon_n)) &\hsp{if $n$ is even and $\ell=\ell^*$} \\
d^{-n}\lambda(t(\varepsilon_n)) &\hsp{if $n$ is odd and $\ell=\ell^*$}\\
0 &\hsp{if $\ell\neq \ell^*$.} 
\end{cases}
$$
\end{defn}

\begin{rem}
The isomorphisms $\varphi_n$ for $n=0,1$ preserve the trace. Moreover, $\tr_{n+1}|_{A_{n}}=\tr_n$ for all $n\in\N$ as $\lambda$ is a Forbenius-Perron eigenvector.
\end{rem}

\begin{prop}[Conditional Expectations]\label{conditional}
If $\ell=[\eta_1\varepsilon_1\varepsilon_2^*\cdots \varepsilon_{2n-1}\varepsilon_{2n}^*\eta_2^*]\in A_n$, the conditional expectation $A_n\to A_{n-1}$ is given by
\begin{align*}
E_{A_{n-1}}(\ell)
&= 
\begin{cases}\D
d^{-1}\delta_{\varepsilon_n,\varepsilon_{n+1}}
\left(\frac{\lambda(s(\varepsilon_n))}{\lambda(t(\varepsilon_n))}\right)
[\eta_1\varepsilon_1\varepsilon_2^*\cdots \varepsilon_{n-1}\varepsilon_{n+2}^*\cdots\varepsilon_{2n-1}\varepsilon_{2n}^*\eta_2^*]& \hsp{$n$ even}\\ \D
d^{-1}\delta_{\varepsilon_n,\varepsilon_{n+1}} 
\left(\frac{\lambda(t(\varepsilon_n))}{\lambda(s(\varepsilon_n))}\right)
[\eta_1\varepsilon_1\varepsilon_2^*\cdots \varepsilon_{n-1}^*\varepsilon_{n+2}\cdots\varepsilon_{2n-1}\varepsilon_{2n}^*\eta_2^*]& \hsp{$n$ odd.}
\end{cases}
\end{align*}
\end{prop}
\begin{proof}
We consider the case $n$ even. The case $n$ odd is similar. We must show $\tr_{n}(xy)=\tr_{n-1}(E_{A_{n-1}}(x)y)$ for all $x\in A_n$ and $y\in A_{n-1}$. It suffices to check when $x,y$ are loops. If
$$
x=[\eta_1\varepsilon_1\varepsilon_2^*\cdots \varepsilon_{2n-1}\varepsilon_{2n}^*\eta_2^*]\hsp{and}
y=[\eta_3\xi_1\xi_2^*\cdots \xi_{2n-3}\xi_{2n-2}^*\eta_4^*],
$$
using the formula above, we have
{\fontsize{10}{10}{
\begin{align*}
\tr_{n-1}(E_{A_{n-1}}(x)y)&= d^{-1}
\delta_{\varepsilon_n,\varepsilon_{n+1}}\delta_{y_{[1,n]},x_{[n+2,2n+2]}^*}
\frac{\lambda(s(\varepsilon_n))}{\lambda(t(\varepsilon_n))}\tr_{n-1}(
[\eta_1\varepsilon_1\cdots \varepsilon_{n-1} \xi_n^*\xi_{n+1}\cdots \xi_{2n-2}^*\eta_4^*])\\
&=d^{-n}\delta_{y_{[1,n]},x_{[n+2,2n+2]}^*}\delta_{\varepsilon_n,\varepsilon_{n+1}}\delta_{x_{[1,n]},y_{[n+1,2n-2]}^*}\lambda(s(\varepsilon_n))
=\tr_n(xy).
\end{align*}}}
\end{proof}

\begin{defn}[Jones Projections]\label{jonesproj}
For $n\in\N$, define distinguished elements of $A_{n+1}$ as follows: if $n$ is odd, define
{\fontsize{10}{10}{
$$
F_n= \sum\limits_{\vec{i}}\sum_{t(\eta)=s(\varepsilon_{i_1})} \frac{[ \lambda(t(\varepsilon_{i_n}))\lambda(t(\varepsilon_{i_{n+1}}))]^{1/2}}{\lambda(s(\varepsilon_{i_n}))} [\eta\varepsilon_{i_1}\varepsilon_{i_2}^*\cdots \varepsilon_{i_{n-1}}^* \varepsilon_{i_n}\varepsilon_{i_n}^*\varepsilon_{i_{n+1}}\varepsilon_{i_{n+1}}^*\varepsilon_{i_{n-1}}\cdots\varepsilon_{i_{2}}\varepsilon_{i_1}^*\eta^*]
$$}}
where the sum is taken over all vectors  $\vec{i}=(i_1,i_2,\dots,i_{n+1})$ such that 
$$
[\varepsilon_{i_1}\varepsilon_{i_2}^*\cdots \varepsilon_{i_{n-1}}^* \varepsilon_{i_n}\varepsilon_{i_n}^*\varepsilon_{i_{n+1}}\varepsilon_{i_{n+1}}^*\varepsilon_{i_{n-1}}\cdots\varepsilon_{i_{2}}\varepsilon_{i_1}^*]\in G_{n+1,+}
$$
If $n$ is even, then define
{\fontsize{10}{10}{
$$
F_n=
\sum\limits_{\vec{i}}\sum\limits_{t(\eta)=s(\varepsilon_{i_1})} \frac{[ \lambda(s(\varepsilon_{i_n}))\lambda(s(\varepsilon_{i_{n+1}}))]^{1/2}}{\lambda(t(\varepsilon_{i_n}))} [\eta\varepsilon_{i_1}\varepsilon_{i_2}^*\cdots \varepsilon_{i_{n-1}} \varepsilon_{i_n}^*\varepsilon_{i_n}\varepsilon_{i_{n+1}}^*\varepsilon_{i_{n+1}}\varepsilon_{i_{n-1}}^*\cdots\varepsilon_{i_{2}}\varepsilon_{i_1}^*\eta^*]
$$}}
with a similar limitation on the vectors $\vec{i}=(i_1,i_2,\dots,i_{j+1})$.
\end{defn}

\begin{lem}\label{basiclemma}
\item[(1)]
$F_n x F_n = dE_{A_{n-1}}(x)F_n$ for all $x\in A_{n}$,
\item[(2)]
The map $A_{n-1}\to A_{n-1}F_n$ by $y\mapsto y F_n$ is injective, and
\item[(3)] $\tr_{n+1}(xF_n)=d^{-1}\tr_n(x)$ for all $x\in A_n$.
\end{lem}
\begin{proof}
We prove the case $n$ odd. The case $n$ even is similar. 
\item[(1)]
If $x=[\eta_1\varepsilon_1\varepsilon_2^*\cdots\varepsilon_{n-1}\varepsilon_n^*\cdots \varepsilon_{2n-1}\varepsilon_{2n}^*\eta_2^*]\in A_{n}$, then
{\fontsize{10}{10}{
\begin{align*}
F_n xF_n &= \sum\limits_{\vec{i}}\sum_{t(\eta)=s(\varepsilon_{i_1})} \frac{[ \lambda(t(\varepsilon_{i_n}))\lambda(t(\varepsilon_{i_{n+1}}))]^{1/2}}{\lambda(s(\varepsilon_{i_n}))} [\eta\varepsilon_{i_1}\varepsilon_{i_2}^*\cdots \varepsilon_{i_{n-1}}^* \varepsilon_{i_n}\varepsilon_{i_n}^*\varepsilon_{i_{n+1}}\varepsilon_{i_{n+1}}^*\varepsilon_{i_{n-1}}\cdots\varepsilon_{i_{2}}\varepsilon_{i_1}^*\eta^*]\times\\
&\hspace{.2in}x\sum\limits_{\vec{j}}\sum_{t(\kappa)=s(\varepsilon_{j_1})} \frac{[\lambda(t(\varepsilon_{j_n}))\lambda(t(\varepsilon_{j_{n+1}}))]^{1/2}}{\lambda(s(\varepsilon_{j_n}))} [\kappa\varepsilon_{j_1}\varepsilon_{j_2}^*\cdots \varepsilon_{j_{n-1}}^* \varepsilon_{j_n}\varepsilon_{j_n}^*\varepsilon_{j_{n+1}}\varepsilon_{j_{n+1}}^*\varepsilon_{j_{n-1}}\cdots\varepsilon_{j_{2}}\varepsilon_{j_1}^*\kappa^*]\\
&=\sum\limits_{s(\varepsilon)=s(\varepsilon_{n-1})}
\frac{[\lambda(t(\varepsilon))\lambda(t(\varepsilon_{{n+1}}))]^{1/2}}{\lambda(s(\varepsilon))} [\eta_1\varepsilon_1\varepsilon_2^*\cdots \varepsilon_{n-1}^*\varepsilon\varepsilon^*\varepsilon_n\varepsilon_{n+1}^*\cdots \varepsilon_{2n-1}\varepsilon_{2n}^*\eta_2^*]\times\\
&\hspace{.2in}\sum\limits_{\vec{j}}\sum_{t(\kappa)=s(\varepsilon_{j_1})} \frac{[\lambda(t(\varepsilon_{j_n}))\lambda(t(\varepsilon_{j_{n+1}}))]^{1/2}}{\lambda(s(\varepsilon_{j_n}))} [\kappa\varepsilon_{j_1}\varepsilon_{j_2}^*\cdots \varepsilon_{j_{n-1}}^* \varepsilon_{j_n}\varepsilon_{j_n}^*\varepsilon_{j_{n+1}}\varepsilon_{j_{n+1}}^*\varepsilon_{j_{n-1}}\cdots\varepsilon_{j_{2}}\varepsilon_{j_1}^*\kappa^*]\\
&=\delta_{\varepsilon_n,\varepsilon_{n+1}}\frac{ \lambda(t(\varepsilon_{n}))}{\lambda(s(\varepsilon_{n}))} \sum\limits_{\substack{s(\varepsilon)=s(\varepsilon_{n-1})\\s(\xi)=s(\varepsilon_{n+2})}}
\frac{[ \lambda(t(\varepsilon))\lambda(t(\xi))]^{1/2}}{\lambda(s(\varepsilon))} [\eta_1\varepsilon_1\varepsilon_2^*\cdots \varepsilon_{n-1}^*\varepsilon\varepsilon^*\xi\xi^*\varepsilon_{n+2}\cdots \varepsilon_{2n-1}\varepsilon_{2n}^*\eta_2^*]\\
&=dE_{A_{n-1}}(x) F_n.
\end{align*} }}
\item[(2)] 
If $y=[\eta_1\varepsilon_1\varepsilon_2^*\cdots\varepsilon_{2n-3}\varepsilon_{2n-2}^*\eta_2^*]$, then 
{\fontsize{10}{10}{
\begin{align*}
yF_n &= y\sum\limits_{\vec{i}}\sum_{t(\eta)=s(\varepsilon_{i_1})} \frac{[\lambda(t(\varepsilon_{i_n}))\lambda(t(\varepsilon_{i_{n+1}}))]^{1/2}}{\lambda(s(\varepsilon_{i_n}))} [\eta\varepsilon_{i_1}\varepsilon_{i_2}^*\cdots \varepsilon_{i_{n-1}}^* \varepsilon_{i_n}\varepsilon_{i_n}^*\varepsilon_{i_{n+1}}\varepsilon_{i_{n+1}}^*\varepsilon_{i_{n-1}}\cdots\varepsilon_{i_{2}}\varepsilon_{i_1}^*\eta^*]\\
&=\sum\limits_{\substack{s(\varepsilon)=s(\varepsilon_{n-1})\\s(\xi)=s(\varepsilon_{n+2})}}
\frac{[ \lambda(t(\varepsilon))\lambda(t(\xi))]^{1/2}}{\lambda(s(\varepsilon))}[\eta_1\varepsilon_1\varepsilon_2^*\cdots \varepsilon_{n-1}^*\varepsilon\varepsilon^*\xi\xi^*\varepsilon_{n+2}\cdots \varepsilon_{2n-1}\varepsilon_{2n}^*\eta_2^*].
\end{align*}}}
Hence it is clear that the standard matrix units of $A_{n-1}$ are mapped to a linearly independent set.\item[(3)]
Another straightforward calculation.
\end{proof}

\begin{prop}[Basic Construction]\label{basic}
For $n\in\N$, the inclusion $A_{n-1}\subset A_n\subset (A_{n+1},\tr_{n+1},d^{-1}F_n)$ is standard. Hence for all $k\geq 0$, there are isomorphisms $\varphi_k\colon A_k\to M_k$ preserving the trace such that $\varphi_{k+1}|_{A_{k}}=\varphi_{k}$ and $\varphi_{m}(F_n)=E_n$ for all $m>n$.
\end{prop}
\begin{proof}
We construct the isomorphisms $\varphi_n$ for $n\geq 1$ by induction on $n$. The base case is finished. Suppose we have constructed $\varphi_n$ for $n\geq 1$. We know that $M_{n+1}=M_n E_n M_n$ and $A_n\cong M_n$ via $\varphi_n$. By Lemmata \ref{basicconstruction} and \ref{basiclemma}, there is an algebra isomorphism $h_{n+1}\colon M_{n+1}=M_nE_nM_n\to A_nF_nA_n\subseteq A_{n+1}$ such that $E_n\mapsto F_n$. But $\dim(M_{n+1})=\dim(A_{n+1})$, so $A_{n+1}=A_n F_n A_n$, and we set $\varphi_{n+1}=h_{n+1}^{-1}$, which extends $\varphi_n$. Finally, note the $\varphi_m$'s preserve the trace by construction and the uniqueness of the Markov trace.
\end{proof}

\subsection{Relative commutants versus loops on $\Gamma$}
We provide isomorphisms between the relative commutants of the tower $(A_n)_{n\geq 0}$ and the spaces $G_{n,\pm}$. 

\begin{prop}[Central Vectors] 
A basis for the central vectors $A_0'\cap A_n$ is given by
$$
S_{0,n}=\set{\sum\limits_{t(\eta)=s(\varepsilon_1)} [\eta\varepsilon_{1}\varepsilon_2^*\cdots \varepsilon_{2n-1}\varepsilon_{2n}^*\eta^*]\in A_n}{[\varepsilon_1\varepsilon_2^*\cdots \varepsilon_{2m-1}\varepsilon_{2m}^*]\in G_{n,+}}.
$$
A basis for the central vectors $A_1'\cap A_{n+1}$ is given by
$$
S_{1,n+1}=\set{\sum\limits_{\substack{t(\eta)=s(\varepsilon)\\t(\varepsilon)=t(\varepsilon_1)}} [\eta\varepsilon\varepsilon_{1}^*\varepsilon_2\cdots \varepsilon_{2n-1}^*\varepsilon_{2n}\varepsilon^*\eta^*]\in A_{n+1}}{[\varepsilon_1^*\varepsilon_2\cdots \varepsilon_{2n-1}^*\varepsilon_{2n}]\in G_{n,-}}.
$$
\end{prop}
\begin{proof}
Note that if $[\zeta_1\zeta_2^*]\in A_0$, then we have
{\fontsize{10}{10}{
\begin{align*}
[\zeta_1\zeta_2^*]\cdot\sum\limits_{t(\eta)=s(\varepsilon_1)}[\eta \varepsilon_1\varepsilon_2^*\cdots \varepsilon_{2n-1}\varepsilon_{2n}^*\eta^*]
&=
\sum\limits_{t(\eta)=s(\varepsilon_1)}\delta_{\zeta_2,\eta}[\zeta_1 \varepsilon_1\varepsilon_2^*\cdots \varepsilon_{2n-1}\varepsilon_{2n}^*\eta^*]\\
&=
[\zeta_1 \varepsilon_1\varepsilon_2^*\cdots \varepsilon_{2n-1}\varepsilon_{2n}^*\zeta_2^*]
=
\sum\limits_{t(\eta)=s(\varepsilon_1)}\delta_{\eta,\zeta_1}[\eta \varepsilon_1\varepsilon_2^*\cdots \varepsilon_{2n-1}\varepsilon_{2n}^*\zeta_2^*]\\
&=
\left(\sum\limits_{t(\eta)=s(\varepsilon_1)}[\eta \varepsilon_1\varepsilon_2^*\cdots \varepsilon_{2n-1}\varepsilon_{2n}^*\eta^*]\right)\cdot [\zeta_1\zeta_2^*]
\end{align*}}}
Hence $S_{0,n}\subset A_0'\cap A_n$. Similarly, $S_{1,n+1}\subset A_1'\cap A_n$.

Suppose now that $x\in A_0'\cap A_n$. Then since
$1_{A_0}=\sum_{\eta} [\eta\eta^*]$,
we have
$$
x=\left(\sum\limits_{\eta} [\eta\eta^*]\right) x =\left(\sum\limits_{\eta} [\eta\eta^*]\cdot[\eta\eta^*]\right) x
=\sum\limits_{\eta} [\eta\eta^*]\cdot x\cdot [\eta\eta^*] \in\spann(S_{0,n}).
$$
Similarly, $A_1'\cap A_{n+1}\subseteq \spann(S_{1,n+1})$.
\end{proof}

\begin{defn}\label{commutant}
For $n\in\Z_{\geq 0}$, let $H_{n,+}=A_0'\cap A_n$, $H_{n,-}=A_1'\cap A_{n+1}$, $Q_{n,+}=M_0'\cap M_n$, and $Q_{n,-}=M_1'\cap M_{n+1}$.
\end{defn}

\begin{cor}\label{gpaspaceiso}
There are canonical algebra $*$-isomorphisms $\phi_{n,\pm}\colon G_{n,\pm}\to H_{n,\pm}$. If $n=0$, the isomorphisms are given by
$$
\phi_{0,+}(v_+)=\sum\limits_{t(\eta)=v_+} [\eta\eta^*]\hsp{and} \phi_{0,-}(v_-)=\sum\limits_{t(\eta)=s(\varepsilon);t(\varepsilon)=v_-} [\eta\varepsilon\varepsilon^*\eta^*].
$$
For $n\in\N$, the isomorphisms are given by
\begin{align*}
\phi_{n,+}([\varepsilon_1\varepsilon_2^*\cdots \varepsilon_{2n-1}\varepsilon_{2n}^*])&= \sum\limits_{t(\eta)=s(\varepsilon_1)}[\eta \varepsilon_1\varepsilon_2^*\cdots \varepsilon_{2n-1}\varepsilon_{2n}^*\eta^*]\hsp{and}\\ 
\phi_{n,-}([\varepsilon_1^*\varepsilon_2\cdots \varepsilon_{2n-1}^*\varepsilon_{2n}])&= \sum\limits_{\substack{t(\eta)=s(\varepsilon)\\t(\varepsilon)=t(\varepsilon_1)}}[\eta\varepsilon \varepsilon_1^*\varepsilon_2\cdots \varepsilon_{2n-1}^*\varepsilon_{2n}\varepsilon^*\eta^*].
\end{align*}
\end{cor}

\begin{rem}\label{multiso}
For $n\geq 0$, $\psi_{n,\pm}=\varphi_n|_{H_{n,\pm}}\circ \phi_{n,\pm}\colon G_{n,\pm}\to Q_{n,\pm}$ are isomorphisms.
\end{rem}

It will be helpful to have an explicit Pimsner-Popa basis for $A_1$ over $A_0$:

\begin{prop}[Pimsner-Popa Bases]\label{basis} For each $v_+\in \V_+$, pick a distinguished $\eta_{v_+}$ with $t(\eta_{v_+})=v_+$. Set
\begin{align*}
B_{1}&=\set{ \left(\frac{d\lambda(s(\varepsilon_2))}{\lambda(t(\varepsilon_2))}\right)^{1/2}\sum\limits_{t(\eta)=s(\varepsilon_{1})}[\eta\varepsilon_1\varepsilon_2^*\eta^*]}{[\varepsilon_1\varepsilon_2^*]\in G_{1,+}}\hsp{and}\\
B_{2}&=\set{\left(\frac{d\lambda(s(\varepsilon_2))}{\lambda(t(\varepsilon_2))}\right)^{1/2}[\eta_1\varepsilon_1\varepsilon_2^*\eta_{s(\varepsilon_2)}^*]}{s(\varepsilon_1)\neq s(\varepsilon_2)}.
\end{align*}
Then $B=B_{1}\amalg B_{2}$ is a Pimsner-Popa basis for $A_1$ over $A_0$.
\end{prop}
\begin{proof}
Suppose $x=[\zeta_1\xi_1\xi_2^* \zeta_2^*]\in A_1$. 
\itt{Case 1} Suppose that $s(\xi_1)=s(\xi_2)$, so $[\xi_1\xi_2^*]\in G_{1,+}$. If $b\in B_{2}$, then
$E_{A_0}(b^*x)=0$ as the formula will have delta functions $\delta_{\xi_i,\varepsilon_i}$ for $i=1,2$.
Hence we have 
{\fontsize{10}{10}{
\begin{align*}
\sum\limits_{b\in B} bE_{A_0}(b^*x) &=\sum\limits_{b\in B_{1}} bE_{A_0}(b^*x)
=\sum\limits_{b\in B_1}\frac{d\lambda(s(\varepsilon_2))}{\lambda(t(\varepsilon_2))}\sum\limits_{\substack{t(\eta)=s(\varepsilon_{1})\\t(\zeta)=s(\varepsilon_{1})}}
[\eta \varepsilon_1\varepsilon_2^* \eta^*] E_{A_0}\big([\zeta\varepsilon_2\varepsilon_1^* \zeta^*]\cdot[\zeta_1\xi_1\xi_2^* \zeta_2^*]\big)\\
&=\sum\limits_{b\in B_1}\frac{d\lambda(s(\varepsilon_2))}{\lambda(t(\varepsilon_2))}\sum\limits_{t(\eta)=s(\varepsilon_{1})} \delta_{\zeta_1,\zeta} \delta_{\xi_1,\varepsilon_1}[\eta \varepsilon_1\varepsilon_2^* \eta^*] E_{A_0}([\zeta\varepsilon_2\xi_2^* \zeta^*_2])\\
&=\sum\limits_{b\in B_1}\frac{d\lambda(s(\varepsilon_2))}{\lambda(t(\varepsilon_2))}\sum\limits_{t(\eta)=s(\xi_{1})} [\eta \xi_1\varepsilon_2^* \eta^*] E_{A_0}([\zeta_1\varepsilon_2\xi_2^* \zeta^*_2])\\
&=\sum\limits_{b\in B_1}\sum\limits_{t(\eta)=s(\xi_{1})} \delta_{\xi_2,\varepsilon_2}[\eta \xi_1\varepsilon_2^* \eta^*] \cdot[\zeta_1 \zeta^*_2]
=[\zeta_1 \xi_1\xi_2^* \zeta_2]=x.
\end{align*}}}
\itt{Case 2}
Suppose that $s(\xi_1)\neq s(\xi_2)$. If $b\in B_{1}$, then similarly,
$E_{A_0}(b^*x)=0$. Hence
{\fontsize{10}{10}{
\begin{align*}
\sum\limits_{b\in B} bE_{A_0}(b^*x) &=\sum\limits_{b\in B_{2}} bE_{A_0}(b^*x)
=\sum\limits_{b\in B_2}\frac{d\lambda(s(\varepsilon_2))}{\lambda(t(\varepsilon_2))}[\eta_1 \varepsilon_1\varepsilon_2^* \eta_{s(\varepsilon_2)}^*] E_{A_0}\big([\eta_{s(\varepsilon_2)}\varepsilon_2\varepsilon_1^* \eta_1^*]\cdot[\zeta_1\xi_1\xi_2^* \zeta_2^*]\big)\\
&= [\zeta_1 \xi_1\xi_2^* \eta_{s(\xi_2)}^*] \cdot[\eta_{s(\xi_2)}\zeta_2^*]=[\zeta_1 \xi_1\xi_2^* \zeta_2]=x.
\end{align*}}}
\end{proof}
\begin{rem}
One could also take
$$
B_{2}=\set{\left(\frac{d\lambda(s(\varepsilon_2))}{m_+(s(\varepsilon_2))\lambda(t(\varepsilon_2))}\right)^{1/2}[\eta_1\varepsilon_1\varepsilon_2^*\eta_2^*]}{s(\varepsilon_1)\neq s(\varepsilon_2)}.
$$
\end{rem}

\begin{cor}[Commutant Conditional Expectations]\label{commutantconditional}
If 
$$x=\sum\limits_{t(\zeta)=s(\xi_1)} [\zeta\xi_1\xi_2^*\cdots \xi_{2n-1}\xi_{2n}^*\zeta^*]\in A_0'\cap A_n,$$ 
the conditional expectation $A_0'\cap A_n\to A_1'\cap A_{n}$ is given by
$$
E_{A_1'}(x)
= 
d^{-1}\delta_{\xi_1,\xi_{2n}}
\left(\frac{\lambda(s(\xi_1))}{\lambda(t(\xi_1))}\right)
\sum\limits_{t(\zeta)=s(\varepsilon); t(\varepsilon)=t(\xi_2)}
[\eta\varepsilon\xi_2^*\xi_3\cdots \xi_{2n-2}^*\xi_{2n-1}\varepsilon^*\eta^*].
$$
\end{cor}
\begin{proof}
Let $B$ be as in \ref{basis}. By \ref{sumb}, we have
$$
d^2E_{A_1'}(x)=\sum\limits_{b\in B} bxb^* =  \sum\limits_{b\in B_1} bxb^*+ \sum\limits_{b\in B_2} bxb^*.
$$
We treat each sum separately:
{\fontsize{10}{10}{
\begin{align*}
\sum\limits_{b\in B_1} bxb^*
&= \sum\limits_{b\in B_1}\left(\frac{d\lambda(s(\varepsilon_2))}{\lambda(t(\varepsilon_2))}\right)\sum\limits_{\substack{t(\eta)=s(\varepsilon_{1})=t(\kappa)\\t(\zeta)=s(\xi_1)}}[\eta\varepsilon_1\varepsilon_2^*\eta^*]\cdot [\zeta\xi_1\xi_2^*\cdots \xi_{2n-1}\xi_{2n}^*\zeta^*]\cdot [\kappa\varepsilon_2\varepsilon_1^*\kappa^*]\\
&= d\sum\limits_{\substack{s(\varepsilon)=s(\varepsilon_2)\\t(\varepsilon)=t(\xi_2)}}\left(\frac{\lambda(s(\varepsilon_2))}{\lambda(t(\varepsilon_2))}\right)\sum\limits_{\substack{t(\eta)=s(\varepsilon_{1})=t(\kappa)\\t(\zeta)=s(\xi_1)}}
\delta_{\eta,\zeta}\delta_{\zeta,\kappa} \delta_{\varepsilon_2,\xi_1}\delta_{\varepsilon_2,\xi_{2n}}
[\eta\varepsilon\xi_2^*\cdots \xi_{2n-1}\varepsilon^*\kappa^*]\\
&= d\sum\limits_{\substack{t(\eta)=s(\varepsilon)=s(\xi_1)\\ t(\varepsilon)=t(\xi_2)}}\left(\frac{\lambda(s(\xi_1))}{\lambda(t(\xi_1))}\right)
\delta_{\xi_1,\xi_{2n}}
[\eta\varepsilon\xi_2^*\cdots \xi_{2n-1}\varepsilon^*\eta^*]
\end{align*}}}
Similarly, we have
$$
\sum\limits_{b\in B_2} bxb^* = d\sum\limits_{\substack{t(\eta)=s(\varepsilon)\neq s(\xi_1)\\ t(\varepsilon)=t(\xi_2)}}\left(\frac{\lambda(s(\xi_1))}{\lambda(t(\xi_1))}\right)
\delta_{\xi_1,\xi_{2n}}
[\eta\varepsilon\xi_2^*\cdots \xi_{2n-1}\varepsilon^*\eta^*].
$$
Putting these two together, we get the desired formula for $E_{A_1'}$.
\end{proof}

\subsection{The bipartite graph planar algebra and the isomorphism}

We refer the reader to \cite{MR1865703} for the full definition of the planar algebra of a bipartite graph.

Let $G_\bullet$ be the planar algebra of the bipartite graph $\Gamma$ with spin vector $\lambda$ as in Subsections \ref{loop} and \ref{towerofloops}. We briefly recall the action of tangles on the $G_{n,\pm}$, and we calculate some necessary examples.

A state $\sigma$ of a tangle $T$ is a way of assigning the regions and strings of $T$ with compatible vertices and edges of $\Gamma$ respectively, i.e., if a string $S$ of $T$ partitions the unshaded region $R_+$ from the shaded region $R_-$, then for $\sigma(S)\in \E$, $s(\sigma(S))=\sigma(R_+)\in \V_+$ and $t(\sigma(S))=\sigma(R_-)\in \V_-$.

Define the output loop $\ell_\sigma$ as the loop obtained by reading clockwise around the outer boundary of $T$ once it has been labeled by $\sigma$.

Suppose now that $T$ has $n$ input disks, and $\ell=\ell_1\otimes\cdots\otimes\ell_n$ is a simple tensor of loops where $\ell_i$ is a loop in $G_{n_i,\pm_i}$. Then the action of $T$ on $\ell$ is given by
$$
T(\ell)=\sum\limits_{\text{states }\sigma} c(\sigma,\ell) \ell_\sigma,
$$
where $c(\sigma,\ell)$ is a correction factor defined as follows:
\be
\item[(1)] First, label the regions and strings of $T$ adjacent to the input disks with the edges and vertices which compose the $\ell_i$'s. If the labeling contradicts $\sigma$, then $c(\sigma,\ell)=0$.
\item[(2)] If the labels agree, put the tangle in a standard form similar to Section \ref{padefn}, where the only difference is that the half the strings emanate from the top of the input rectangles, and half the strings emanate down, but the $*$ is still on the left side. Let $E(T)$ be the set of local extrema of the strings of the standard form of the tangle. For each $e\in E(T)$, let $\conv(e)$ be the vertex assigned by $\sigma$ to the convex region of the extrema, and let $\conc(e)$ be the vertex assigned to the concave region. Set
$$
k_e=\sqrt{\frac{\lambda(\conv(e))}{\lambda(\conc(e))}}.
$$
Below is an example of an extrema $e$ on a string $S$ with $\sigma(S)=\varepsilon$, connecting vertices $w,v$:
$$
\input{pictures/extremum}.
$$
Note that $\conv(e)$ may be in either $\V_+$ or $\V_-$. Finally, set
$$
c(\sigma,\ell)=\prod\limits_{e\in E(T)} k_e.
$$
\ee
The $*$-structure on the bipartite graph planar algebra is given as follows: if $T$, $\ell$ are as above, then
$$
T(\ell_1^*\otimes\cdots\otimes \ell_n^*)=T^*(\ell_1\otimes\cdots\otimes \ell_n)^*
$$
where $T^*$ is the mirror image of $T$, and the adjoint of a loop is the loop traversed backwards as in Definition \ref{adjoint}.

\begin{rem}
Closed, contractible strings are traded for a multiplicative factor of $d$ as $\lambda$ is a Frobenius-Perron eigenvector (see Definition \ref{fp}).
\end{rem}

\begin{rem}\label{commutantinclude}
Note from Corollary \ref{gpaspaceiso} that there is a natural inclusion identification $G_{n,-}\to G_{n+1,+}$ given by
$$
[\varepsilon_1^*\varepsilon_2\cdots\varepsilon_{2n-1}^*\varepsilon_{2n}]\longmapsto\sum\limits_{t(\varepsilon)=s(\varepsilon_1)}[\varepsilon\varepsilon_1^*\varepsilon_2\cdots\varepsilon_{2n-1}^*\varepsilon_{2n}\varepsilon^*].
$$
\end{rem}

\begin{exs}\label{gpa}
\item[(0)] If $\ell_1,\ell_2\in G_{n,\pm}$, then $\D\ell_1\cdot\ell_2=\multiplication{\ell_1}{\ell_2}\hs$, the shading depending on $n,\pm$.
\item[(1)] For $n\in\N$ odd, 
$$
\jonesprojection= 
{\fontsize{10}{10}{
\sum\limits_{\vec{i}}\frac{[ \lambda(t(\varepsilon_{i_n}))\lambda(t(\varepsilon_{i_{n+1}}))]^{1/2}}{\lambda(s(\varepsilon_{i_n}))} [\varepsilon_{i_1}\cdots \varepsilon_{i_{n-1}}^* \varepsilon_{i_n}\varepsilon_{i_n}^*\varepsilon_{i_{n+1}}\varepsilon_{i_{n+1}}^*\varepsilon_{i_{n-1}}\cdots\varepsilon_{i_1}^*]},}
$$
where the sum is taken over all vectors  $\vec{i}=(i_1,i_2,\dots,i_{n+1})$ such that 
$$
[\varepsilon_{i_1}\varepsilon_{i_2}^*\cdots \varepsilon_{i_{n-1}}^* \varepsilon_{i_n}\varepsilon_{i_n}^*\varepsilon_{i_{n+1}}\varepsilon_{i_{n+1}}^*\varepsilon_{i_{n-1}}\cdots\varepsilon_{i_{2}}\varepsilon_{i_1}^*]\in G_{n+1,+}.
$$
There is a similar formula for $n$ even. (Compare with Definition \ref{jonesproj}.)
\item[(2)] Suppose $\ell=[\varepsilon_1\varepsilon_2^*\cdots\varepsilon_{2n-1}\varepsilon_{2n}^*]\in G_{n,+}$.
\be
\item[(i)]
If $n$ is even, then 
$$\conditional{\ell}=\delta_{\varepsilon_n,\varepsilon_{n+1}}\frac{\lambda(s(\varepsilon_n))}{\lambda(t(\varepsilon_n))}[\varepsilon_1\varepsilon_2^*\cdots \varepsilon_{n-1}\varepsilon_{n+2}^*\varepsilon_{2n-1}\cdots\varepsilon_{2n}^*],
$$
with a similar formula for $n$ odd. (Compare with Proposition \ref{conditional}.)
\item[(ii)] If $n$ is even, then
$$
\inclusion{\ell}=\sum\limits_{s(\varepsilon)=s(\varepsilon_n)} [\varepsilon_1\varepsilon_2^*\cdots\varepsilon_n^*\varepsilon\varepsilon^*\varepsilon_{n+1} \cdots\varepsilon_{2n-1}\varepsilon_{2n}^*],
$$
with a similar formula for $n$ odd. (Compare with Definition \ref{inclusions}.)
\item[(iii)] 
$\D
\commutant{\ell}=\delta_{\varepsilon_1,\varepsilon_{2n}}\frac{\lambda(s(\varepsilon_1))}{\lambda(t(\varepsilon_1))}[\varepsilon_2^*\varepsilon_3\cdots \varepsilon_{2n-2}^*\varepsilon_{2n-1}].
$\\
(Compare with Proposition \ref{commutantconditional} and Remark \ref{commutantinclude}.)
\ee
\item[(3)] 
If $\ell=[\varepsilon_1^*\varepsilon_2\cdots\varepsilon_{2n-1}^*\varepsilon_{2n}]\in G_{n,-}$, then
$$
\minusinclusion{\ell}=\sum\limits_{t(\varepsilon)=s(\varepsilon_1)}[\varepsilon\varepsilon_1^*\varepsilon_2\cdots\varepsilon_{2n-1}^*\varepsilon_{2n}\varepsilon^*],
$$
which may be identified with $\ell\in G_{n+1,+}$ by Remark \ref{commutantinclude}.
\end{exs}

\begin{thm}
The canonical relative commutant planar $*$-algebra $P_\bullet$ is isomorphic to the bipartite graph planar $*$-algebra $G_\bullet$ of the Bratteli diagram $\Gamma$ for the inclusion $M_0\subset M_1$.
\end{thm}
\begin{proof}
Under the $*$-algebra isomorphisms $\theta_n^{-1}\circ \psi_{n,+}\colon G_{n,+}\to P_{n,+}$ and $\theta_{n+1}^{-1}\circ\psi_{n,-}\colon G_{n,-}\to P_{n,-}$, we may transport the graph planar algebra structure of $G_\bullet$ to the algebras $P_{n,\pm}$. The result now follows by Lemma \ref{technical} and Examples \ref{gpa}.
\end{proof}

\section{The Embedding Theorem}\label{sec:embed}
Let $Q_\bullet$ be a finite depth subfactor planar algebra of modulus $d$. Pick $r\geq 0$ minimal such that $Q_{2r,+}\subset Q_{2r+1,+}\subset (Q_{2r+2,+},e_{2r+1})$ is standard (with the usual trace), and set $s=2r$ (this is possible if and only if $Q_\bullet$ is of finite depth). In fact, $Q_{k,+}\subset Q_{k+1,+}\subset (Q_{k+2,+},e_{k+1})$ is standard for all $k\geq s$. For $n\geq 0$, set $M_n=Q_{s+n,+}$, $F_{n+1}=E_{s+n+1}$ (shifted Jones projections), 
\begin{align*}
P_{n,+}&=M_0'\cap M_n = Q_{s,+}'\cap Q_{s+n,+},\hsp{and}\\
P_{n,-}&=M_1'\cap M_{n+1} = Q_{s+1,+}'\cap Q_{s+n+1,+}.
\end{align*}
Then $P_\bullet$ has a canonical planar algebra structure (we suppress the isomorphisms $\theta_n$ with the tensor products of $Q_{s+1,+}$ over $Q_{s,+}$).

\begin{thm}\label{embed}
Define $\Phi\colon Q_\bullet\to P_\bullet$ by adding $s=2r$ strings to the left for $x\in Q_{n,+}$ and adding $s+1$ strings to the left for $x\in Q_{n,-}$.
\input{pictures/embed}
Then $\Phi$ is an inclusion of planar $*$-algebras.
\end{thm}
\begin{proof} We use Lemma \ref{technical}. Note that $\Phi(x^*)=\Phi(x)^*$ and $\Phi(xy)=\Phi(x) \Phi(y)$ for all $x,y\in Q_{n,\pm}$.
\item[(1)] Since $\Phi(E_j)=E_{s+j}=F_j$ for all $j\in\N$, we have $\Phi(E_j  x)=F_j\Phi(x)$ and $\Phi(x E_j)=\Phi(x) F_j$ for all $x\in Q_{n,\pm}$ and all $j\in\N$.
\item[(2)] Note that
\be
\item[(i)] For $n\in \N$, $\Phi(E_{Q_{n-1,+}}(x))= E_{P_{n-1,+}}(\Phi(x))$ since 
$$E_{Q_{s+n-1,+}}|_{Q_{s,+}'\cap Q_{s+n,+}} = E_{Q_{s+n-1,+}}|_{P_{n,+}}=E_{P_{n-1,+}}$$ 
(since $Q_{s,+}\subset Q_{s+n-1,+}$, we have that $E_{Q_{s+n-1,+}}$ preserves $Q_{s,+}$-central vectors as it is $Q_{s+n-1,+}$-bilinear).
\item[(ii)] $\Phi( \beta_{n+1}(x))= \beta_{n+1}(\Phi(x))$ for all $x\in Q_{n,+}$ since the inclusion $P_{n,+}\to P_{n+1,+}$ is the restriction of the inclusion $Q_{s+n,+}\to Q_{s+n+1,+}$.
\item[(iii)] Let $B=\{b\}$ be a Pimsner-Popa basis for $M_1= Q_{s+1,+}$ over $M_{0}=Q_{s,+}$. Since each $b\in B$ is an $(s+1,+)$-box in $Q_{s+1,+}$, 
$$\frac{1}{d}
\sum\limits_{b\in B}\input{pictures/ppbasis}=\sum\limits_{b\in B} be_{s+1}b^* = 1_{P_{s+2}} = \input{pictures/identity}\,.
$$
Then by Proposition \ref{sumb} and Theorem \ref{canonical}, for all $x\in Q_{n,+}$,
\ee
\begin{align*}
\gamma_n^+(\Phi(x))&=\frac{1}{d}\sum\limits_{b\in B} b\Phi(x)b^*=\frac{1}{d}
\sum\limits_{b\in B}\input{pictures/LCworks1}\\
&=\frac{1}{d}\sum\limits_{b\in B}\input{pictures/LCworks2}
=\input{pictures/LCworks3}
=\Phi(\gamma_n^+(x)).
\end{align*}

\item[(3)] The inclusion $i^-_n\colon P_{n,-}\to P_{n+1,+}$ is the identity by the canonical structure of the relative commutant planar algebra. If $x\in Q_{n,-}$, then we have
$$
i^-_n(\Phi(x))=\Phi(x)=\input{pictures/minusincludeworks}=\Phi(i^-_n(x)).
$$
\end{proof}

\begin{cor}
Let $N\subset M$ be a finite index, finite depth $II_1$-subfactor, and let $P_\bullet$ be the canonical associated subfactor planar algebra. Let $\Gamma$ be the principal graph of $N\subset M$, and let $G_\bullet$ be the bipartite graph planar algebra of $\Gamma$. Then there is an embedding of planar algebras $P_\bullet\to G_\bullet$.
\end{cor}

\bibliographystyle{amsalpha}
\bibliography{bibliography}
\end{document}

%% file: pictures/operators.tex
\tikzstyle{shaded}=[fill=red!10!blue!20!gray!30!white]
\tikzstyle{unshaded}=[fill=white]
\tikzstyle{empty box}=[circle, draw, thick, fill=white, opaque, inner sep=2mm]
\tikzstyle{annular}=[scale=.7, inner sep=1mm, baseline]
\tikzstyle{rectangular}=[scale=.75, inner sep=1mm, baseline]

\newcommand{\conditional}[1]{
\begin{tikzpicture}[rectangular]
	\clip (2,1) --(-2,1) -- (-2,-1) -- (2,-1);
	\filldraw[shaded] (-1.2,1.5)--(-1.2,-1.5)--(-.8,-1.5)--(-.8,1.5);
	\draw (.6,1.5)--(.6,-1.5);
	\draw (1,.5) arc (180:0:.3cm) -- (1.6,-.5);
	\draw (1,-.5) arc (-180:0:.3cm);
	\draw [ultra thick] (2,1) --(-2,1) -- (-2,-1) -- (2,-1)--(2,1);
	\filldraw[thick, unshaded] (1.3,.5) --(1.3,-.5) -- (-1.4,-.5) -- (-1.4,.5)--(1.3,.5);
	\node at (-.1,.75) {$\cdots$};
	\node at (-.1,-.75) {$\cdots$};
	\node at (-.05,0) {$#1$};
\end{tikzpicture}}

\newcommand{\commutant}[1]{
\begin{tikzpicture}[rectangular]
	\clip (2,1) --(2,-1) -- (-2,-1) -- (-2,1);
	\filldraw[shaded] (-2.5,1.5)--(-2.5,-1.5)--(-.6,-1.5)--(-.6,1.5);
	\draw (1.3,1.5)--(1.3,-1.5);
	\filldraw[unshaded] (-1,.5) arc (0:180:.3cm) -- (-1.6,-.5) arc (-180:0:.3cm) -- (-1,.5);
	\draw [ultra thick] (2,1) --(-2,1) -- (-2,-1) -- (2,-1)--(2,1);
	\filldraw[thick, unshaded] (1.5,.5) --(1.5,-.5) -- (-1.2,-.5) -- (-1.2,.5)--(1.5,.5);
	\node at (.4,.75) {$\cdots$};
	\node at (.4,-.75) {$\cdots$};
	\node at (.15,0) {$#1$};
\end{tikzpicture}}

\newcommand{\minuscommutant}[1]{
\begin{tikzpicture}[rectangular]
	\clip (2,1) --(2,-1) -- (-2,-1) -- (-2,1);
	\filldraw[unshaded] (-2.5,1.5)--(-2.5,-1.5)--(-.6,-1.5)--(-.6,1.5);
	\draw (1.3,1.5)--(1.3,-1.5);
	\filldraw[shaded] (-1,.5) arc (0:180:.3cm) -- (-1.6,-.5) arc (-180:0:.3cm) -- (-1,.5);
	\draw [ultra thick] (2,1) --(-2,1) -- (-2,-1) -- (2,-1)--(2,1);
	\filldraw[thick, unshaded] (1.5,.5) --(1.5,-.5) -- (-1.2,-.5) -- (-1.2,.5)--(1.5,.5);
	\node at (.4,.75) {$\cdots$};
	\node at (.4,-.75) {$\cdots$};
	\node at (.15,0) {$#1$};
\end{tikzpicture}}

\newcommand{\inclusion}[1]{
\begin{tikzpicture}[rectangular]
	\clip (2,1) --(-2,1) -- (-2,-1) -- (2,-1);
	\filldraw[shaded] (-1.2,1.5)--(-1.2,-1.5)--(-.8,-1.5)--(-.8,1.5);
	\draw (1.6,1.5)--(1.6,-1.5);
	\draw (1,1.5) --(1,-1.5);
	\draw [ultra thick] (2,1) --(-2,1) -- (-2,-1) -- (2,-1)--(2,1);
	\filldraw[thick, unshaded] (1.2,.5) --(1.2,-.5) -- (-1.4,-.5) -- (-1.4,.5)--(1.2,.5);
	\node at (0.1,.75) {$\cdots$};
	\node at (0.1,-.75) {$\cdots$};
	\node at (-0.1,0) {$#1$};
\end{tikzpicture}}

\newcommand{\minusinclusion}[1]{
\begin{tikzpicture}[rectangular]
	\clip (2,1) --(-2,1) -- (-2,-1) -- (2,-1);
	\filldraw[shaded] (-1.6,1.5)--(-1.6,-1.5)--(-1,-1.5)--(-1,1.5);
	\draw (1.2,1.5)--(1.2,-1.5);
	\draw [ultra thick] (2,1) --(-2,1) -- (-2,-1) -- (2,-1)--(2,1);
	\filldraw[thick, unshaded] (-1.2,.5) --(-1.2,-.5) -- (1.4,-.5) -- (1.4,.5)--(-1.2,.5);
	\node at (0.2,.75) {$\cdots$};
	\node at (0.2,-.75) {$\cdots$};
	\node at (0.1,0) {$#1$};
\end{tikzpicture}}

\newcommand{\commutantinclusion}[1]{
\begin{tikzpicture}[rectangular]
	\clip (2,1) --(-2,1) -- (-2,-1) -- (2,-1);
	\filldraw[shaded] (-1.6,1.5)--(-1.6,-1.5)--(-2.8,-1.5)--(-2.8,1.5);
	\draw (1.2,1.5)--(1.2,-1.5);
	\draw (-1,1.5) --(-1,-1.5);
	\draw [ultra thick] (2,1) --(-2,1) -- (-2,-1) -- (2,-1)--(2,1);
	\filldraw[thick, unshaded] (-1.2,.5) --(-1.2,-.5) -- (1.4,-.5) -- (1.4,.5)--(-1.2,.5);
	\node at (0.2,.75) {$\cdots$};
	\node at (0.2,-.75) {$\cdots$};
	\node at (0.1,0) {$#1$};
\end{tikzpicture}}

\newcommand{\rotation}[1]{
\begin{tikzpicture}[rectangular]
	\clip (2.5,1.5) --(-2.5,1.5) -- (-2.5,-1.5) -- (2.5,-1.5);
	\filldraw[shaded] (-1,1.5)--(-1,0)--(-.6,0)--(-.6,1.5);
	\draw (-.2,1.5)--(-.2,-1.5);
	\draw (.5,1.5) --(.5,-1.5);
	\draw (.8,0)--(.8,-1.5);
	\draw (1.2,0) --(1.2,-1.5);
	\filldraw[shaded] (-2,1.5)--(-2,-.5) arc (-180:0:.7cm) --(-1,-.5) arc (0:-180:.3) -- (-1.6,2);
	\draw (.8,.5) arc (180:0:.7cm) --(2.2,-1.5) --(1.8,-1.5)--(1.8,.5) arc (0:180:.3);
	\draw [ultra thick] (2.5,1.5) --(-2.5,1.5) -- (-2.5,-1.5) -- (2.5,-1.5)--(2.5,1.5);
	\filldraw[thick, unshaded] (1.4,.5) --(1.4,-.5) -- (-1.2,-.5) -- (-1.2,.5)--(1.4,.5);
	\node at (.2,1) {$\cdots$};
	\node at (.2,-1) {$\cdots$};
	\node at (0,0) {$#1$};
	\node at (-2.5,1) [right] {$*$};
	\node at (-1.1,0) [left] {$*$};
\end{tikzpicture}}

\newcommand{\minusrotation}[1]{
\begin{tikzpicture}[rectangular]
	\clip (2.5,1.5) --(-2.5,1.5) -- (-2.5,-1.5) -- (2.5,-1.5);
	\filldraw[shaded] (-3,1.5)-- (-.2,1.5)--(-.2,-1.5)--(-3,-1.5)--(-3,1.5);
	\filldraw[unshaded] (-1,1.5)--(-1,0)--(-.6,0)--(-.6,1.5);
	\draw (.5,1.5) --(.5,-1.5);
	\draw (.8,0)--(.8,-1.5);
	\draw (1.2,0) --(1.2,-1.5);
	\filldraw[unshaded] (-2,1.5)--(-2,-.5) arc (-180:0:.7cm) --(-1,-.5) arc (0:-180:.3) -- (-1.6,2);
	\draw (.8,.5) arc (180:0:.7cm) --(2.2,-1.5) --(1.8,-1.5)--(1.8,.5) arc (0:180:.3);
	\draw [ultra thick] (2.5,1.5) --(-2.5,1.5) -- (-2.5,-1.5) -- (2.5,-1.5)--(2.5,1.5);
	\filldraw[thick, unshaded] (1.4,.5) --(1.4,-.5) -- (-1.2,-.5) -- (-1.2,.5)--(1.4,.5);
	\node at (.2,1) {$\cdots$};
	\node at (.2,-1) {$\cdots$};
	\node at (0,0) {$#1$};
	\node at (-2.5,1) [right] {$*$};
	\node at (-1.1,0) [left] {$*$};
\end{tikzpicture}}

\newcommand{\rotationbyone}[1]{
\begin{tikzpicture}[rectangular]
	\clip (2,1) --(2,-1) -- (-2,-1) -- (-2,1);
	\filldraw[shaded] (-2.5,1.5)--(-2.5,-1.5)--(-.6,-1.5)--(-.6,1.5);
	\draw (.6,1.5)--(.6,-1.5);
	\draw (1,-1.5)--(1,.5) arc (180:0:.3cm) -- (1.6,-1.5);
	\filldraw[unshaded] (-1,-.5) arc (0:-180:.3cm) -- (-1.6,1.5) --(-1,2)--(-1,-.5);
	\draw [ultra thick] (2,1) --(-2,1) -- (-2,-1) -- (2,-1)--(2,1);
	\filldraw[thick, unshaded] (1.2,.5) --(1.2,-.5) -- (-1.2,-.5) -- (-1.2,.5)--(1.2,.5);
	\node at (0,.75) {$\cdots$};
	\node at (0,-.75) {$\cdots$};
	\node at (0,0) {$#1$};
\end{tikzpicture}}

\newcommand{\minusrotationbyone}[1]{
\begin{tikzpicture}[rectangular]
	\clip (2,1) --(2,-1) -- (-2,-1) -- (-2,1);
	\filldraw[unshaded] (-2.5,1.5)--(-2.5,-1.5)--(-.6,-1.5)--(-.6,1.5);
	\draw (.6,1.5)--(.6,-1.5);
	\draw (1,-1.5)--(1,.5) arc (180:0:.3cm) -- (1.6,-1.5);
	\filldraw[shaded] (-1,-.5) arc (0:-180:.3cm) -- (-1.6,1.5) --(-1,2)--(-1,-.5);
	\draw [ultra thick] (2,1) --(-2,1) -- (-2,-1) -- (2,-1)--(2,1);
	\filldraw[thick, unshaded] (1.2,.5) --(1.2,-.5) -- (-1.2,-.5) -- (-1.2,.5)--(1.2,.5);
	\node at (0,.75) {$\cdots$};
	\node at (0,-.75) {$\cdots$};
	\node at (0,0) {$#1$};
\end{tikzpicture}}

\newcommand{\multiplication}[2]{
\begin{tikzpicture}[rectangular]
	\clip (2,1.75) --(2,-1.75) -- (-2,-1.75) -- (-2,1.75);
	\draw (1.3,2)--(1.3,-2);
	\draw (.9,2)--(.9,-2);
	\draw (-1.3,2)--(-1.3,-2);
	\draw (-.9,2)--(-.9,-2);
	\draw [ultra thick] (2,1.75) --(-2,1.75) -- (-2,-1.75) -- (2,-1.75)--(2,1.75);
	\filldraw[thick, unshaded] (1.5,.25) --(1.5,1.25) -- (-1.5,1.25) -- (-1.5,.25)--(1.5,.25);
	\filldraw[thick, unshaded] (1.5,-.25) --(1.5,-1.25) -- (-1.5,-1.25) -- (-1.5,-.25)--(1.5,-.25);
	\node at (0,0) {$\cdots$};
	\node at (0,-1.5) {$\cdots$};
	\node at (0,1.5) {$\cdots$};
	\node at (0,-.75) {$#2$};
	\node at (0,.75) {$#1$};
\end{tikzpicture}}

\newcommand{\jonesprojection}{
\begin{tikzpicture}[rectangular]
	\clip (2,1) --(-2,1) -- (-2,-1) -- (2,-1);
	\filldraw[shaded] (-1.4,1)--(-1.4,-1)--(-1,-1)--(-1,1);
	\draw (.4,1)--(.4,-1);
	\draw (.8,1) arc (-180:0:.4cm);
	\draw (.8,-1) arc (180:0:.4cm);
	\draw [ultra thick] (2,1) --(-2,1) -- (-2,-1) -- (2,-1)--(2,1);	
	\node at (-.5,-.3) {$\underbrace{\qquad \qquad}_{n-1}$};
	\node at (-.3,.4) {$\cdots$};
\end{tikzpicture}}

\newcommand{\minusjonesprojection}{
\begin{tikzpicture}[rectangular]
	\clip (2,1) --(-2,1) -- (-2,-1) -- (2,-1);
	\filldraw[shaded] (-1.4,1)--(-1.4,-1)--(-3,-1)--(-3,1);
	\draw (-1,1)--(-1,-1);
	\draw (.4,1)--(.4,-1);
	\draw (.8,1) arc (-180:0:.4cm);
	\draw (.8,-1) arc (180:0:.4cm);
	\draw [ultra thick] (2,1) --(-2,1) -- (-2,-1) -- (2,-1)--(2,1);
	\node at (-.5,-.3) {$\underbrace{\qquad \qquad}_{n-2}$};
	\node at (-.3,.4) {$\cdots$};
\end{tikzpicture}}

\newcommand{\multistep}[2]{
\begin{tikzpicture}[rectangular]
	\clip (2,1.25) --(2,-1.25) -- (-2,-1.25) -- (-2,1.25);
	\filldraw[shaded] (-1.7,2)--(-1.7,-2)--(-1.4,-2)--(-1.4,2);	
	\draw (-.6,2)--(-.6,-2);
	\draw (.3,1.25) arc (-180:0:.4cm);
	\draw (-.3,1.25) arc (-180:0:1cm);
	\draw (.3,-1.25) arc (180:0:.4cm);
	\draw (-.3,-1.25) arc (180:0:1cm);
	\draw [ultra thick] (2,1.25) --(-2,1.25) -- (-2,-1.25) -- (2,-1.25)--(2,1.25);
	\node at (-1.15,-.5) {$\underbrace{\qquad }_{\hspace{.09in}#1}$};
	\node at (.05,.85) {$\underbrace{\hspace{.2in}}_{\hspace{.07in}#2}$};
	\node at (-.95,0) {$\cdots$};
	\node at (.7,.7) {{\fontsize{10}{10}{$\vdots$}}};
	\node at (.7,-.4) {{\fontsize{10}{10}{$\vdots$}}};
\end{tikzpicture}}

%% file: pictures/embed.tex
$$\begin{tikzpicture}[baseline]
	\clip (1.5,1.5) --(-1.5,1.5) -- (-1.5,-1.5) -- (1.5,-1.5);
	\filldraw[shaded] (-.8,1.5)--(-.8,-1.5)--(-.4,-1.5)--(-.4,1.5);
	\draw  (.8,-1.5)--(.8,1.5);
	\draw [ultra thick] (1.5,1.5) --(-1.5,1.5) -- (-1.5,-1.5) -- (1.5,-1.5)--(1.5,1.5);
	\filldraw[thick, unshaded] (1,.5) --(-1,.5) -- (-1,-.5) -- (1,-.5)--(1,.5);
	\node at (0,-1.1) {$\underbrace{\qquad \qquad}_{n}$};
	\node at (.2,1) {$\cdots$};
	\node at (.2,-.75) {$\cdots$};
	\node at (0,0) {$x$};
\end{tikzpicture}\longmapsto
\begin{tikzpicture}[baseline]
	\clip (1.5,1.5) --(-3.5,1.5) -- (-3.5,-1.5) -- (1.5,-1.5);
	\filldraw[shaded] (-3,1.5)--(-3,-1.5)--(-2.6,-1.5)--(-2.6,1.5);
	\filldraw[shaded] (-1.8,1.5)--(-1.8,-1.5)--(-1.4,-1.5)--(-1.4,1.5);	
	\filldraw[shaded] (-.8,1.5)--(-.8,-1.5)--(-.4,-1.5)--(-.4,1.5);
	\draw  (.8,-1.5)--(.8,1.5);
	\draw [ultra thick] (1.5,1.5) --(-3.5,1.5) -- (-3.5,-1.5) -- (1.5,-1.5)--(1.5,1.5);
	\filldraw[thick, unshaded] (1,.5) --(-1,.5) -- (-1,-.5) -- (1,-.5)--(1,.5);
	\node at (-2.2,-1.1) {$\underbrace{\qquad \qquad}_{s=2r}$};
	\node at (0,-1.1) {$\underbrace{\qquad \qquad}_{n}$};
	\node at (.2,1) {$\cdots$};
	\node at (.2,-.75) {$\cdots$};
	\node at (-2.2,0) {$\cdots$};
	\node at (0,0) {$x$};
\end{tikzpicture}$$

%% file: pictures/exampletangle.tex
\begin{tikzpicture}[rectangular]
	\clip (6,6) --(-6,6) -- (-6,-6) -- (6,-6);
	\filldraw[shaded] (-3.5,6.5)--(-3.5,4)--(-3,4)--(-3,5) arc (180:0:.5cm)--(-2,4)-- (-1.5, 4)--(-1.5,4.5) arc (180:0:.5cm) 
	--(-.5,-3.5)--(-2,-3.5)--(-2,-3) arc (0:180:1cm) --(-4,-3) 
.. controls ++(270:5cm) and ++(270:5cm) .. 
	(5.5,0)
	-- (5.5,2) arc (0:180:1cm)--(3.5,2) --( 3,2)--(3,6.5);
	\filldraw[unshaded] (1, 6.5)--(1,-.5)--(3,-.5)--(3,0) .. controls ++(90:1cm) and ++(270:1cm) ..  (5,1.3)--(5,2) arc (0:180:.5cm) -- (4,1.5) -- (2,1.5)--(2,6.5);
	\filldraw[unshaded] (-2.5, -3.5)--(-2.5,-3) arc(0:180:.5cm) 
	.. controls ++(270:3cm) and ++(270:3cm) .. 
	(2,-2.5) arc (0:180:1cm)  -- (0,-3.5);	
	\filldraw[shaded] (-5.25,6.5) 
	.. controls ++(270:7cm) and ++(90:7cm) .. 
	(-1.5,-3.5)--(-1,-3.5) 
	.. controls ++(90:7cm) and ++(270:7cm) .. 
	(-4.75,6.5);
	\draw [ultra thick] (6,6) --(-6,6) -- (-6,-6) -- (6,-6)--(6,6);
	\draw[thick, unshaded] (-4,4.5)--(-4,3.5) --(-1,3.5) -- (-1,4.5)--(-4,4.5);
	\draw[thick, unshaded] (-3, -3)--(1, -3) --(1,-4) -- (-3,-4)--(-3,-3);
	\draw[thick, unshaded] (.5,0)--(3.5,0) --(3.5,-1) -- (.5,-1)--(.5,0);
	\draw[thick, unshaded] (1.5,2)--(4.5,2) --(4.5,1) -- (1.5,1)--(1.5,2);
	\draw[dashed] (-6,.25)-- (6,.25); 
\end{tikzpicture}

%% file: pictures/shadedmin.tex
\begin{tikzpicture}[rectangular]
	\clip (1,1) --(1,-1) -- (-1,-1) -- (-1,1);
	\draw[shaded] (.75,1.5)--(.75,-1.5)--(1.25,-1.5)--(1.25,1.5);
	\draw[shaded] (-.75,1.5)--(-.75,-1.5)--(-1.25,-1.5)--(-1.25,1.5);
	\draw[shaded] (-.5,1.5)--(-.5,.25) arc (180:360:.5cm)--(.5,1.5);	
	\draw[ultra thick] (1,1) --(1,-1) -- (-1,-1) -- (-1,1)--(1,1);	
	\draw[dashed] (-1,-.5)-- (1,-.5); 
\end{tikzpicture}\rightarrow
\begin{tikzpicture}[rectangular]
	\clip (1,1) --(1,-1) -- (-1,-1) -- (-1,1);
	\draw[shaded] (.75,1.5)--(.75,-1.5)--(1.25,-1.5)--(1.25,1.5);
	\draw[shaded] (-.75,1.5)--(-.75,-1.5)--(-1.25,-1.5)--(-1.25,1.5);
	\draw[shaded] (-.5,1.5)--(-.5,.25) arc (180:360:.5cm)--(.5,1.5);	
	\draw[ultra thick] (1,1) --(1,-1) -- (-1,-1) -- (-1,1)--(1,1);	
	\draw[dashed] (-1,.5)-- (1,.5); 
\end{tikzpicture}

%% file: pictures/unshadedmin.tex
\begin{tikzpicture}[rectangular]
	\clip (1,1) --(1,-1) -- (-1,-1) -- (-1,1);
	\draw[shaded] (2,2) --(2,-2) -- (-2,-2) -- (-2,2);
	\draw[unshaded] (.75,1.5)--(.75,-1.5)--(1.25,-1.5)--(1.25,1.5);
	\draw[unshaded] (-.75,1.5)--(-.75,-1.5)--(-1.25,-1.5)--(-1.25,1.5);
	\draw[unshaded] (-.5,1.5)--(-.5,.25) arc (180:360:.5cm)--(.5,1.5);	
	\draw[ultra thick] (1,1) --(1,-1) -- (-1,-1) -- (-1,1)--(1,1);	
	\draw[dashed] (-1,-.5)-- (1,-.5); 
\end{tikzpicture}\rightarrow
\begin{tikzpicture}[rectangular]
	\clip (1,1) --(1,-1) -- (-1,-1) -- (-1,1);
	\draw[shaded] (2,2) --(2,-2) -- (-2,-2) -- (-2,2);
	\draw[unshaded] (.75,1.5)--(.75,-1.5)--(1.25,-1.5)--(1.25,1.5);
	\draw[unshaded] (-.75,1.5)--(-.75,-1.5)--(-1.25,-1.5)--(-1.25,1.5);
	\draw[unshaded] (-.5,1.5)--(-.5,.25) arc (180:360:.5cm)--(.5,1.5);	
	\draw[ultra thick] (1,1) --(1,-1) -- (-1,-1) -- (-1,1)--(1,1);	
	\draw[dashed] (-1,.5)-- (1,.5); 
\end{tikzpicture}

%% file: pictures/shadedmax.tex
\begin{tikzpicture}[rectangular]
	\clip (1,1) --(1,-1) -- (-1,-1) -- (-1,1);
	\draw[shaded] (.75,1.5)--(.75,-1.5)--(1.25,-1.5)--(1.25,1.5);
	\draw[shaded] (-.75,1.5)--(-.75,-1.5)--(-1.25,-1.5)--(-1.25,1.5);
	\draw[shaded] (-.5,-1.5)--(-.5,-.25) arc (180:0:.5cm)--(.5,-1.5);	
	\draw[ultra thick] (1,1) --(1,-1) -- (-1,-1) -- (-1,1)--(1,1);	
	\draw[dashed] (-1,-.5)-- (1,-.5); 
\end{tikzpicture}\rightarrow
\begin{tikzpicture}[rectangular]
	\clip (1,1) --(1,-1) -- (-1,-1) -- (-1,1);
	\draw[shaded] (.75,1.5)--(.75,-1.5)--(1.25,-1.5)--(1.25,1.5);
	\draw[shaded] (-.75,1.5)--(-.75,-1.5)--(-1.25,-1.5)--(-1.25,1.5);
	\draw[shaded] (-.5,-1.5)--(-.5,-.25) arc (180:0:.5cm)--(.5,-1.5);	
	\draw[ultra thick] (1,1) --(1,-1) -- (-1,-1) -- (-1,1)--(1,1);	
	\draw[dashed] (-1,.5)-- (1,.5); 
\end{tikzpicture}

%% file: pictures/unshadedmax.tex
\begin{tikzpicture}[rectangular]
	\clip (1,1) --(1,-1) -- (-1,-1) -- (-1,1);
	\draw[shaded] (2,2) --(2,-2) -- (-2,-2) -- (-2,2);
	\draw[unshaded] (.75,1.5)--(.75,-1.5)--(1.25,-1.5)--(1.25,1.5);
	\draw[unshaded] (-.75,1.5)--(-.75,-1.5)--(-1.25,-1.5)--(-1.25,1.5);
	\draw[unshaded] (-.5,-1.5)--(-.5,-.25) arc (180:0:.5cm)--(.5,-1.5);	
	\draw[ultra thick] (1,1) --(1,-1) -- (-1,-1) -- (-1,1)--(1,1);	
	\draw[dashed] (-1,-.5)-- (1,-.5); 
\end{tikzpicture}\rightarrow
\begin{tikzpicture}[rectangular]
	\clip (1,1) --(1,-1) -- (-1,-1) -- (-1,1);
	\draw[shaded] (2,2) --(2,-2) -- (-2,-2) -- (-2,2);
	\draw[unshaded] (.75,1.5)--(.75,-1.5)--(1.25,-1.5)--(1.25,1.5);
	\draw[unshaded] (-.75,1.5)--(-.75,-1.5)--(-1.25,-1.5)--(-1.25,1.5);
	\draw[unshaded] (-.5,-1.5)--(-.5,-.25) arc (180:0:.5cm)--(.5,-1.5);	
	\draw[ultra thick] (1,1) --(1,-1) -- (-1,-1) -- (-1,1)--(1,1);	
	\draw[dashed] (-1,.5)-- (1,.5); 
\end{tikzpicture}

%% file: pictures/alphas.tex
\begin{tikzpicture}[annular]
	\clip (0,0) circle (2cm);
	\draw[ultra thick] (0,0) circle (2cm);	
	\node at (0,0)  [empty box] (T) {};
	\node at (95:1.60cm) [right] {$*$};
	\filldraw[shaded] (100:4cm)--(0,0)--(140:4cm);	
	\filldraw[shaded] (220:4cm)--(0,0)--(260:4cm);          
	\filldraw[shaded] (T.-40) .. controls ++(-40:11mm) and ++           (40:11mm) .. (T.40);
	\draw[ultra thick] (0,0) circle (2cm);
	\node at (0,0)  [empty box] (T) {};
	\node at (180:1cm) {$\cdot$};
	\node at (160:1cm) {$\cdot$};
	\node at (200:1cm) {$\cdot$};
	\node at (T.70) [above] {$*$};		
\end{tikzpicture}\hs,\hs
\begin{tikzpicture}[annular]
	\clip (0,0) circle (2cm);
	\draw[ultra thick] (0,0) circle (2cm);	
	\node at (0,0)  [empty box] (T) {};
	\node at (95:1.60cm) [right] {$*$};
	\filldraw[shaded] (100:4cm)--(0,0)--(140:4cm);	
	\filldraw[shaded] (220:4cm)--(0,0)--(260:4cm);             
	\filldraw[shaded] (0,0) -- (60:3cm) -- (0:4cm) -- (-60:3cm) -- (0,0) .. controls ++(-30:2cm) and ++(30:2cm) .. (0,0);		
	\draw[ultra thick] (0,0) circle (2cm);
	\node at (0,0)  [empty box] (T) {};
	\node at (180:1cm) {$\cdot$};
	\node at (160:1cm) {$\cdot$};
	\node at (200:1cm) {$\cdot$};
	\node at (T.70) [above] {$*$};		
\end{tikzpicture}\hs,\cdots,\hs
\begin{tikzpicture}[annular]
	\clip (0,0) circle (2cm);
	\draw[ultra thick] (0,0) circle (2cm);	
	\node at (0,0)  [empty box] (T) {};
	\node at (180:1.90cm) [right] {$*$};
	\filldraw[shaded] (310:4cm)--(0,0)--(350:4cm);	
	\filldraw[shaded] (190:4cm)--(0,0)--(230:4cm);            
	\filldraw[shaded] (0,0) -- (30:3cm) -- (90:4cm) -- (150:3cm) -- (0,0) .. controls ++(120:2cm) and ++(60:2cm) .. (0,0);	
	\draw[ultra thick] (0,0) circle (2cm);
	\node at (0,0)  [empty box] (T) {};
	\node at (270:1cm) {$\cdot$};
	\node at (290:1cm) {$\cdot$};
	\node at (250:1cm) {$\cdot$};
	\node at (T.90) [above] {$*$};	
\end{tikzpicture}

%% file: pictures/betas.tex
\begin{tikzpicture}[annular]
	\clip (0,0) circle (2cm);
	\draw[ultra thick] (0,0) circle (2cm);	
	\node at (0,0)  [empty box] (T) {};
	\node at (135:1.80cm) [right] {$*$};
	\filldraw[shaded] (-20:4cm)--(0,0)--(20:4cm);	
	\filldraw[shaded] (160:4cm)--(0,0)--(200:4cm);
	\filldraw[shaded] (250:4cm)--(0,0)--(290:4cm);		             
	\filldraw[shaded] (70:2cm) .. controls ++(250:7mm) and ++(290:7mm) ..         (110:2cm) -- (110:3cm);
	\draw[ultra thick] (0,0) circle (2cm);
	\node at (0,0)  [empty box] (T) {};
	\node at (-135:1cm) {$\cdot$};
	\node at (-145:1cm) {$\cdot$};
	\node at (-125:1cm) {$\cdot$};
	\node at (T.90) [above] {$*$};		
\end{tikzpicture}\hs,\hs
\begin{tikzpicture}[annular]
	\clip (0,0) circle (2cm);
	\draw[ultra thick] (0,0) circle (2cm);	
	\node at (0,0)  [empty box] (T) {};
	\node at (120:1.80cm) [right] {$*$};
	\filldraw[shaded] (160:4cm)--(0,0)--(200:4cm);
	\filldraw[shaded] (250:4cm)--(0,0)--(290:4cm);             
	\filldraw[shaded] (25:2cm) .. controls ++(205:7mm) and ++(245:7mm) .. (65:2cm) -- (65:3cm)--(90:3cm) -- (90:2cm) --(0,0) -- (0:2cm) -- (0:3cm);
	\draw[ultra thick] (0,0) circle (2cm);
	\node at (0,0)  [empty box] (T) {};
	\node at (-135:1cm) {$\cdot$};
	\node at (-145:1cm) {$\cdot$};
	\node at (-125:1cm) {$\cdot$};
	\node at (T.120) [above] {$*$};		
\end{tikzpicture}\hs,\cdots,\hs
\begin{tikzpicture}[annular]
	\clip (0,0) circle (2cm);
	\draw[ultra thick] (0,0) circle (2cm);	
	\node at (0,0)  [empty box] (T) {};
	\node at (135:1.90cm) [right] {$*$};
	\filldraw[shaded] (-20:4cm)--(0,0)--(20:4cm);
	\filldraw[shaded] (250:4cm)--(0,0)--(290:4cm);           
	\filldraw[shaded] (115:2cm) .. controls ++(295:7mm) and ++(335:7mm) .. (155:2cm) -- (155:3cm)--(180:3cm) -- (180:2cm) --(0,0) -- (90:2cm) -- (90:3cm);		
	\draw[ultra thick] (0,0) circle (2cm);
	\node at (0,0)  [empty box] (T) {};
	\node at (315:1cm) {$\cdot$};
	\node at (305:1cm) {$\cdot$};
	\node at (325:1cm) {$\cdot$};
	\node at (T.-155)[below] {$*$};	
\end{tikzpicture}

%% file: pictures/extremum.tex
\comment{\begin{tikzpicture}[rectangular]
	\node at (0,0){w};
	\draw[ultra thick] (.2,0)--(1.8,0);
	\node at (2,0){v};
	\node at (1,.3){$\varepsilon$};
\end{tikzpicture}
\hspace{1in}}
\begin{tikzpicture}[rectangular]
	\draw (-2,-.75) .. controls ++(90:1.5cm) and ++(90:1.5cm) ..  (2,-.75);	
	\node at (0,-.5){convex};
	\node at (0,0){w};
	\node at (0,1.25){concave};
	\node at (0,.75){v};
	\node at (-2,0){$\varepsilon$};
\end{tikzpicture}\longrightarrow k_e=\sqrt{\frac{\lambda(w)}{\lambda(v)}}

%% file: pictures/ppbasis.tex
\begin{tikzpicture}[rectangular]
	\clip  (-2.25,1.8)--(1, 1.8) -- (1,-1.8) -- (-2.25,-1.8) -- (-2.25,1.8); 
	\filldraw[shaded] (-1.8,3)--(-1.8,-3)--(-1.4,-3)--(-1.4,3);
	\filldraw[shaded] (-.6,3)--(-.6,-3)--(-.2,-3)--(-.2,3);
	\filldraw[shaded] (.2,3)--(.2,.5) arc (-180:0:.3cm) --(.8,2) arc (180:0:.3cm) --(1.4,-2) arc (0:-180:.3cm) --(.8,-.5) arc (0:180:.3cm) -- (.2,-3) -- (1.8, -3) -- (1.8,3);
	\draw (3,3)--(3,-3);
	\draw [ultra thick] (3.7,2.5) --(-2.5,2.5) -- (-2.5,-2.5) -- (3.7,-2.5)--(3.7,2.5);
	\draw[thick, unshaded] (-2, 1.5)--(.4, 1.5) --(.4,.5) -- (-2,.5)--(-2,1.5);
	\draw[thick, unshaded] (-2, -1.5)--(.4, -1.5) --(.4,-.5) -- (-2,-.5)--(-2,-1.5);
	\draw[thick, unshaded] (1.2,.5)--(3.2,.5) --(3.2,-.5) -- (1.2,-.5)--(1.2,.5);
	\draw[dashed] (-2.25,1.8)--(1, 1.8) -- (1,-1.8) -- (-2.25,-1.8) -- (-2.25,1.8); 
	\node at (-.8,1){$b$};
	\node at (-.8,-1){$b^*$};
	\node at (2.2,0){$x$};
	\node at (-.95,2.2) {$\cdots$};
	\node at (-.95,-2.2) {$\cdots$};
	\node at (-.95,0) {$\cdots$};
	\node at (2.4,1.5) {$\cdots$};
	\node at (2.4,-1.5) {$\cdots$};
\end{tikzpicture}

%% file: pictures/identity.tex
\begin{tikzpicture}[rectangular]
	\clip (-2.25,1.8)--(.8, 1.8) -- (.8,-1.8) -- (-2.25,-1.8) -- (-2.25,1.8); 
	\filldraw[shaded] (-1.8,3)--(-1.8,-3)--(-1.5,-3)--(-1.5,3);
	\filldraw[shaded] (-.4,3)--(-.4,-3)--(-.1,-3)--(-.1,3);
	\filldraw[shaded] (.2,3)--(.2,-3) --(1.6,-3)--(1.6,3);
	\filldraw[unshaded] (.6,2) arc (180:0:.3cm) --(1.2,-2) arc (0:-180:.3cm) --(.6,2);	\draw (2.8,3)--(2.8,-3);
	\draw [ultra thick] (3.5,2.5) --(-2.5,2.5) -- (-2.5,-2.5) -- (3.5,-2.5)--(3.5,2.5);
	\draw[unshaded] (1,.5)--(3,.5) --(3,-.5) -- (1,-.5)--(1,.5);
	\draw[dashed] (-2.25,1.8)--(.8, 1.8) -- (.8,-1.8) -- (-2.25,-1.8) -- (-2.25,1.8); 
	\node at (-.8,.5) {$\underbrace{\qquad \qquad}_{s+1\hs}$};	
	\node at (2,0){$x$};
	\node at (-.95,-.5) {$\cdots$};
	\node at (2.2,1.5) {$\cdots$};
	\node at (2.2,-1.5) {$\cdots$};
\end{tikzpicture}

%% file: pictures/LCworks1.tex
\begin{tikzpicture}[rectangular]
	\clip (2.5,2.5) --(-2.5,2.5) -- (-2.5,-2.5) -- (2.5,-2.5);
	\filldraw[shaded] (-1.8,3)--(-1.8,-3)--(-1.4,-3)--(-1.4,3);
	\filldraw[shaded] (-.6,3)--(-.6,-3)--(-.2,-3)--(-.2,3);
	\filldraw[shaded] (.2,3)--(.2,-3)--(.6,-3)--(.6,3);
	\draw (1.8,3)--(1.8,-3);
	\draw [ultra thick] (2.5,2.5) --(-2.5,2.5) -- (-2.5,-2.5) -- (2.5,-2.5)--(2.5,2.5);
	\draw[thick, unshaded] (-2, 2)--(.4, 2) --(.4,1) -- (-2,1)--(-2,2);
	\draw[thick, unshaded] (-2, -2)--(.4, -2) --(.4,-1) -- (-2,-1)--(-2,-2);
	\draw[thick, unshaded] (0,.5)--(2,.5) --(2,-.5) -- (0,-.5)--(0,.5);
	\node at (-1,.5) {$\underbrace{\qquad \qquad}_{s}$};	
	\node at (1,-1) {$\underbrace{\qquad \qquad}_{n}$};
	\node at (-.8,1.5){$b$};
	\node at (-.8,-1.5){$b^*$};
	\node at (1,0){$x$};
	\node at (-.95,2.2) {$\cdots$};
	\node at (-.95,-2.2) {$\cdots$};
	\node at (-.95,-.5) {$\cdots$};
	\node at (1.2,1.5) {$\cdots$};
	\node at (1.2,-2) {$\cdots$};
\end{tikzpicture}

%% file: pictures/LCworks2.tex
\begin{tikzpicture}[rectangular]
	\clip (3.7,2.5) --(-2.5,2.5) -- (-2.5,-2.5) -- (3.7,-2.5);
	\filldraw[shaded] (-1.8,3)--(-1.8,-3)--(-1.4,-3)--(-1.4,3);
	\filldraw[shaded] (-.6,3)--(-.6,-3)--(-.2,-3)--(-.2,3);
	\filldraw[shaded] (.2,3)--(.2,.5) arc (-180:0:.3cm) --(.8,2) arc (180:0:.3cm) --(1.4,-2) arc (0:-180:.3cm) --(.8,-.5) arc (0:180:.3cm) -- (.2,-3) -- (1.8, -3) -- (1.8,3);
	\draw (3,3)--(3,-3);
	\draw [ultra thick] (3.7,2.5) --(-2.5,2.5) -- (-2.5,-2.5) -- (3.7,-2.5)--(3.7,2.5);
	\draw[thick, unshaded] (-2, 1.5)--(.4, 1.5) --(.4,.5) -- (-2,.5)--(-2,1.5);
	\draw[thick, unshaded] (-2, -1.5)--(.4, -1.5) --(.4,-.5) -- (-2,-.5)--(-2,-1.5);
	\draw[thick, unshaded] (1.2,.5)--(3.2,.5) --(3.2,-.5) -- (1.2,-.5)--(1.2,.5);
	\draw[dashed] (-2.25,1.8)--(1, 1.8) -- (1,-1.8) -- (-2.25,-1.8) -- (-2.25,1.8); 
	\node at (-.8,1){$b$};
	\node at (-.8,-1){$b^*$};
	\node at (2.2,0){$x$};
	\node at (-.95,2.2) {$\cdots$};
	\node at (-.95,-2.2) {$\cdots$};
	\node at (-.95,0) {$\cdots$};
	\node at (2.4,1.5) {$\cdots$};
	\node at (2.4,-1.5) {$\cdots$};
\end{tikzpicture}

%% file: pictures/LCworks3.tex
\begin{tikzpicture}[rectangular]
	\clip (3.5,2.5) --(-2.5,2.5) -- (-2.5,-2.5) -- (3.5,-2.5);
	\filldraw[shaded] (-1.8,3)--(-1.8,-3)--(-1.5,-3)--(-1.5,3);
	\filldraw[shaded] (-.4,3)--(-.4,-3)--(-.1,-3)--(-.1,3);
	\filldraw[shaded] (.2,3)--(.2,-3) --(1.6,-3)--(1.6,3);
	\filldraw[unshaded] (.6,2) arc (180:0:.3cm) --(1.2,-2) arc (0:-180:.3cm) --(.6,2);	\draw (2.8,3)--(2.8,-3);
	\draw [ultra thick] (3.5,2.5) --(-2.5,2.5) -- (-2.5,-2.5) -- (3.5,-2.5)--(3.5,2.5);
	\draw[thick, unshaded] (1,.5)--(3,.5) --(3,-.5) -- (1,-.5)--(1,.5);
	\draw[dashed] (-2.25,1.8)--(.8, 1.8) -- (.8,-1.8) -- (-2.25,-1.8) -- (-2.25,1.8); 
	\node at (-.8,.5) {$\underbrace{\qquad \qquad}_{s+1\hs}$};	
	\node at (2,0){$x$};
	\node at (-.95,-.5) {$\cdots$};
	\node at (2.2,1.5) {$\cdots$};
	\node at (2.2,-1.5) {$\cdots$};
\end{tikzpicture}

%% file: pictures/minusincludeworks.tex
\begin{tikzpicture}[rectangular]
	\clip (3,1.5) --(-2.5,1.5) -- (-2.5,-1.5) -- (3,-1.5);
	\filldraw[shaded] (-1.8,3)--(-1.8,-3)--(-1.5,-3)--(-1.5,3);
	\filldraw[shaded] (-.4,3)--(-.4,-3)--(-.1,-3)--(-.1,3);
	\filldraw[shaded] (.2,3)--(.2,-3) --(.7,-3)--(.7,3);
	\draw [ultra thick] (3,1.5) --(-2.5,1.5) -- (-2.5,-1.5) -- (3,-1.5)--(3,1.5);
	\draw (2.3,3)--(2.3,-3);
	\draw[thick, unshaded] (.5,.5)--(2.5,.5) --(2.5,-.5) -- (.5,-.5)--(.5,.5);
	\node at (-.8,-.5) {$\underbrace{\qquad \qquad}_{s+1\hs}$};	
	\node at (1.5,0){$x$};
	\node at (-.95,.5) {$\cdots$};
	\node at (1.5,1) {$\cdots$};
	\node at (1.5,-1) {$\cdots$};
\end{tikzpicture}